\documentclass[11pt]{article}
\pdfoutput=1
\usepackage[top=1in, right=1in, left=1in, bottom=1in]{geometry}
\usepackage{amsthm}
\usepackage[hidelinks]{hyperref}

\usepackage{thmtools}
\usepackage{thm-restate}

\theoremstyle{plain}
\newtheorem{theorem}{Theorem}
\newtheorem*{theorem*}{Theorem}
\newtheorem{lemma}{Lemma}
\newtheorem{proposition}{Proposition}
\newtheorem*{lemma*}{Lemma}
\newtheorem{corollary}[theorem]{Corollary}

\renewenvironment{proof}{\noindent{\bf 
Proof.}\hspace*{1em}}{\qed\medskip\\}

\makeatletter
\long\def\@makecaption#1#2{
  \vskip 0.8ex
  \setbox\@tempboxa\hbox{\small {\bf #1:} #2}
  \parindent 1.5em  %
  \dimen0=\hsize
  \advance\dimen0 by -3em
  \ifdim \wd\@tempboxa >\dimen0
  \hbox to \hsize{
    \parindent 0em
    \hfil 
    \parbox{\dimen0}{\def\baselinestretch{0.96}\small
      {\bf #1.} #2
    } 
    \hfil}
  \else \hbox to \hsize{\hfil \box\@tempboxa \hfil}
  \fi
}
\makeatother %
\usepackage[numbers,square]{natbib}

\usepackage{enumerate}
\usepackage{amssymb}
\usepackage{amsbsy}
\usepackage{amsmath}
\usepackage{amsfonts}
\usepackage{latexsym}
\usepackage{graphicx}
\usepackage{color}
\usepackage{xifthen}
\usepackage{xspace}

\usepackage{algorithm}
\usepackage[noend]{algpseudocode}

\usepackage{multirow}
\usepackage[normalem]{ulem}
\usepackage{accents}

\usepackage{array}
\usepackage[T1]{fontenc}
\usepackage{thmtools}
\usepackage{thm-restate}

\usepackage{etoolbox}
\newtoggle{restatements}
\newcommand{\restate}[1]{\iftoggle{restatements}{#1}{}}
\newtoggle{heavyplots}
\newcommand{\heavyplot}[1]{\iftoggle{heavyplots}{#1}{}}

\usepackage{xfrac}

\newcommand{\mc}[1]{\mathcal{#1}}

\newcommand{\wt}[1]{\widetilde{#1}}  %
\newcommand{\wb}[1]{\overline{#1}} %

\newcommand{\norm}[1]{\left\|{#1}\right\|} %
\newcommand{\norms}[1]{\|{#1}\|} %

\newcommand{\opnorm}[1]{\norm{#1}_{2}}  %
\newcommand{\R}{\mathbb{R}} %
\newcommand{\N}{\mathbb{N}} %
\renewcommand{\P}{\mathbb{P}}	%
\newcommand{\I}{\mathbb{I}} %

\providecommand{\argmin}{\mathop{\rm argmin}}

\providecommand{\diag}{\mathop{\rm diag}}

\newcommand{\hinge}[1]{\left[{#1}\right]_+} %

\providecommand{\minimize}{\mathop{\rm minimize}}

\newcommand{\ceil}[1]{\left\lceil{#1}\right\rceil}

\newcommand{\half}{\frac{1}{2}}

\newcommand{\defeq}{\triangleq}
\newcommand{\eqdef}{\triangleq}

\newcommand{\grad}{\nabla}
\newcommand{\hess}{\nabla^2}

\def\ie{\emph{i.e.\ }}

\newcommand{\T}{T} %
\newcommand{\eps}{\varepsilon}
\newcommand{\gap}{\mathsf{gap}}
\newcommand{\del}{\partial}

\newcommand{\s}{x_{\star}}

\newcommand{\Ls}{L_{\star}}
\newcommand{\As}{A_{\star}}

\newcommand{\rs}{\rho \norm{\s}}
\newcommand{\ns}{\norm{\s}} %
\newcommand{\nss}{\norm{\s}^2} %
\newcommand{\IndNC}{}
\newcommand{\gammaplus}{\gamma_+}

\newcommand{\tgrow}{\tau_{\mathrm{grow}}}
\newcommand{\tconv}{\tau_{\mathrm{conv}}}
\newcommand{\tiltgrow}{\tilde{\tau}_{\mathrm{grow}}}
\newcommand{\tiltconv}{\tilde{\tau}_{\mathrm{conv}}}

\newcommand{\sigbar}{\wb{\sigma}}

\newcommand{\smallval}{\nu}

\newcommand{\CP}{Cauchy-point}
\newcommand{\SP}{Solve-problem}
\newcommand{\SSP}{Solve-subproblem}
\newcommand{\SFSP}{Solve-final-subproblem}
\newcommand{\Kprog}{K_{\mathrm{prog}}}
\newcommand{\Kmax}{K_{\max}}
\newcommand{\slb}{r}
\newcommand{\siter}{\Delta^\star}

\newcommand{\glb}{\underline{g}}
\newcommand{\kf}{{K}}
\newcommand{\yout}{y_\mathrm{out}}
\newcommand{\epsgrad}{\eps_{\mathrm{g}}}

\newcommand{\callCP}[1]{\hyperref[func:CP]{\Call{\CP}{#1}}}
\newcommand{\callSP}[1]{\hyperref[func:SP]{\Call{\SP}{#1}}}
\newcommand{\callSSP}[1]{\hyperref[func:SSP]{\Call{\SSP}{#1}}}
\newcommand{\callSFSP}[1]{\hyperref[func:SFSP]{\Call{\SFSP}{#1}}}

\newcommand{\supind}[1]{^{\left({#1}\right)}}
\newcommand{\opt}{^\star}

\newtheorem{assumption}{Assumption}

\numberwithin{theorem}{section}
\numberwithin{lemma}{section}
\numberwithin{claim}{section}
\numberwithin{proposition}{section}
\numberwithin{corollary}{section}

\makeatletter
\let\c@proposition\c@theorem
\let\c@corollary\c@theorem
\let\c@lemma\c@theorem
\let\c@claim\c@theorem
\makeatother

\title{Gradient Descent Finds the 
    Cubic-Regularized Non-Convex Newton Step}

\author{Yair Carmon ~~~ John C.\ Duchi \\
  Stanford University \\
\texttt{\{\href{mailto:yairc@stanford.edu}{yairc}
	,\href{mailto:jduchi@stanford.edu}{jduchi}\}@stanford.edu}}

\date{}

\toggletrue{heavyplots}

\begin{document}
\maketitle

\begin{abstract}
  We consider the minimization of non-convex quadratic forms regularized by 
  a cubic term, which exhibit multiple saddle points and poor local minima. 
  Nonetheless, we prove that, under mild assumptions, gradient descent 
  approximates the \emph{global minimum} to within $\eps$ accuracy in 
  $O(\eps^{-1}\log(1/\eps))$ steps for large $\eps$ and $O(\log(1/\eps))$ 
  steps for small $\eps$ (compared to a condition number we define), with at 
  most logarithmic dependence on the problem dimension. When we use 
  gradient descent to approximate the cubic-regularized 
  Newton step, our result implies a rate of convergence to second-order 
  stationary points of general smooth non-convex functions.
\end{abstract}

\section{Introduction}

We study the optimization problem 
\begin{equation}
  \minimize_{x\in\R^{d}}~
  f\left(x\right) \defeq \half x^{\T}Ax
  +b^{\T}x+\frac{\rho}{3}\left\Vert x\right\Vert ^{3},
  \label{eq:problem}
\end{equation}
where the matrix $A$ is symmetric and possibly indefinite. The problem
\eqref{eq:problem} arises in Newton's method with cubic regularization, for 
(approximately) minimizing a general smooth function $g$. The method 
consists of the iterative 
procedure %
\begin{equation}\label{eq:nesterov-polyak}
y_{t+1}=y_t + \argmin_{x\in\R^{d}}\left\{ \nabla 
g(y_{t})^{\T}x+\half x^{\T}\nabla^{2}g(y_{t})x
+\frac{\rho_t}{3}\norm{x} ^{3}\right\},
\end{equation}
where every iteration requires solution of a problem of the
form~\eqref{eq:problem} and choice of the parameter $\rho_t$.
\citet{Griewank81} first proposed the 
scheme~\eqref{eq:nesterov-polyak}
 (in 
a more general setting), and then
\citet{NesterovPo06} and~\citet{WeiserDeEr07}
independently rediscovered it.
Cubic regularization methods, as well as the closely related trust-region 
methods, are among the 
most practically successful and 
theoretically sound approaches to non-convex 
optimization~\cite{ConnGoTo00, NesterovPo06, CartisGoTo11}. Indeed,
\citet{NesterovPo06} establish that 
$O(\epsilon^{-3/2})$ iterations of the form~\eqref{eq:nesterov-polyak} 
suffice to find an $\epsilon$-second-order-stationary point of $g$, 
meaning a 
point $y_\epsilon$ such that $\norm{\grad g(y_\epsilon)}\le\epsilon$ and 
$\lambda_{\min}(\hess g(y_\epsilon)) \gtrsim -\sqrt{\epsilon}$. However, 
this complexity 
guarantee does not account for the computational cost of solving 
subproblems 
of the form~\eqref{eq:problem}. 

In this work, we study what is perhaps the simplest algorithm for
approximately solving the problem~\eqref{eq:problem}: gradient descent. Each
iteration of gradient descent consists of the transformation $x\mapsto
x-\eta \grad f(x) = x-\eta(Ax + b + \rho \norm{x} x)$ for a step-size
$\eta\in\R$. Thus, the computational cost of a gradient descent iteration is
essentially that of multiplying the matrix $A$ with a vector. Iterative 
methods requiring only matrix-vector products are called 
\emph{matrix-free}, %
and are especially appealing in the setting when
$d$ is large and $A$ has structure, such as sparsity
(cf.~\cite{TrefethenBa97}), which enables efficient computation of $Ax$.
Notably, when $A$ is a Hessian as in~\eqref{eq:nesterov-polyak}, it is often
possible to compute $Ax$ in time linear in $d$~\cite{Pearlmutter94,
  Schraudolph02}, comparable to the time to evaluate a
gradient.

We do not claim that gradient descent is the most efficient method for
solving problem~\eqref{eq:problem}. Indeed, popular matrix-free Krylov
subspace solvers~\cite{CartisGoTo11} provide faster convergence by
definition, as the first $k$ iterates of gradient descent lie in the Krylov
subspace of order $k$, $\mathrm{span}\{b, Ab, \ldots, A^{k-1}b\}$. Moreover,
two-term recursions such as the heavy-ball method~\cite{Polyak64} and
Nesterov's accelerated gradient descent~\cite{Nesterov83} outperform
gradient descent in convex problems, with results extending to several
non-convex scenarios~\cite{CarmonDuHiSi18}. Yet we
believe gradient descent---as a workhorse for numerous large-scale
problems---is a valuable object of study, for the following reasons.

\begin{enumerate}[(1)]
\item By proving concrete upper bounds on the number of gradient steps 
  required to achieve an $\eps$-accurate solution to the 
  problem~\eqref{eq:problem}, we obtain a benchmark for more 
  sophisticated algorithms, such as Krylov subspace methods,
  and providing dimension-independent guarantees on the number 
  of matrix-vector products such methods require to 
  solve~\eqref{eq:problem} to $\eps$ accuracy (see further discussion in 
  \S\ref{sec:intro-related} below).
\item Analysis of optimization methods operating on convex quadratic 
  objectives provides important insight about the performance of these 
  methods for general nonlinear objectives close to a local minimum. 
  Analogously, we believe that analyzing gradient descent  
  on the simple structured non-convex objective~\eqref{eq:problem} will 
  provide useful intuition about the way gradient descent generally 
  navigates saddle 
  points. 
  We show that saddle points may cause gradient descent to stall, but that 
  the overall effect of this stalling on the rate of convergence is bounded, 
  and that the presence of non-convexity slows convergence by at 
  most a logarithmic factor. We expect a similar qualitative picture to 
  emerge for other non-convex problems.
\item Unlike more sophisticated methods, gradient descent (with 
  properly chosen step sizes) is often 
  effective in the stochastic setting where only a noisy estimate of the 
  gradient is available. This effectiveness is well-understood for convex 
  objectives~\cite{Duchi18,BottouCuNo18},
  and it extends to several non-convex problems (notably neural network
  training), for reasons we do not fully  
  understand~\cite{LeCunBeHi15}. Analyzing gradient descent on the 
  non-stochastic problem~\eqref{eq:problem} is a first step towards 
  understanding stochastic gradient descent methods beyond convex
  problems, which may prove useful for stochastic variants of 
  the cubic-regularized Newton's method~\eqref{eq:nesterov-polyak} as 
  well a broader theory of non-convex optimization with stochastic 
  gradient methods.
\end{enumerate}

\subsection{Outline of our contribution}\label{sec:intro-outline}

We begin our development in Section~\ref{sec:prel} with a number of
definitions and results, specifying our assumptions,
characterizing the solution to problem~\eqref{eq:problem}, and proving that 
gradient descent converges to the \emph{global minimum} of
$f$. Additionally, we show that gradient descent
produces iterates with monotonically increasing norm. This  
property is essential to our results, and we use it extensively throughout the 
paper.

In Section~\ref{sub:statement} we provide non-asymptotic rates of
convergence for gradient descent, which are our main results: gradient
descent finds a point $x$ such that $f(x)\le \inf_{\s\in\R^d}f(\s) + \eps$,
in a number of steps that scales as $\log\frac{1}{\eps}$ for
well-conditioned problems and $\frac{1}{\eps}\log\frac{1}{\eps}$ for
poorly-condition problems (for a condition number we define explicitly).  We
outline our proofs in Section~\ref{sec:proofs}, deferring technical
arguments to appendices as necessary. Our first convergence guarantee
includes the term $\log(1/|v_1^{\T}b|)$, where $v_1$ is the eigenvector
corresponding to the smallest eigenvalue of $A$. When 
$v_1^{\T}b=0$---as happens in the so-called ``hard case'' for 
non-convex quadratic problems~\cite{ConnGoTo00}---this term becomes 
infinite. Nevertheless, by applying gradient descent on a slightly perturbed 
problem we achieve convergence rates scaling no worse than 
logarithmically in problem dimension, for any value of $v_1^{\T}b$.
Our results have close connections with the
convergence rates of gradient descent on smooth convex functions and of 
the power method, which we discuss in Section~\ref{sec:discussion}.

We illustrate our results with a number of experiments, which we
report in Section~\ref{sec:experiments}. We explore the trajectory of gradient descent on 
non-convex problem instances, demonstrating its dependence on problem 
conditioning and the presence of saddle points. We then illustrate our 
convergence rate guarantees by running gradient descent over an ensemble 
of random problem instances. This experiment suggests the sharpness of 
our theoretical analysis.

In Section~\ref{sec:linesearch} we extend our scope to step sizes
chosen by exact line search.
If the search is unconstrained,
the method may fail to converge to the global minimum,
but success is guaranteed for a guarded variation of exact line 
search.  Unfortunately, we have
thus far been unable to give rates of convergence for this scheme,
though its empirical behavior is at least as strong as standard gradient
descent.

As our initial motivation for solving problem~\eqref{eq:problem} is the 
regularized Newton's method \eqref{eq:nesterov-polyak}, in
Section~\ref{sec:trustregion} we consider a method for minimizing a general 
non-convex function $g$, which approximates the
iterations~\eqref{eq:nesterov-polyak} via gradient descent. In keeping with  
the theoretical focus of this work, the method is not designed to be 
efficient in practice, but rather to showcase how our analysis applies in 
the context of subproblem solutions. 
  When $g$ has
$2\rho$-Lipschitz continuous Hessian, we show that this
method finds a point $y_\epsilon$ such that $\norm{\nabla g(y_\epsilon)} 
\le \epsilon$ and $\lambda_{\min}(\hess g(y_\epsilon)) \ge  
-\sqrt{\rho\epsilon}$, in $\epsilon^{-2}$ gradient and 
Hessian-vector product evaluations (ignoring constant and logarithmic 
terms), 
which is the rate for gradient descent applied
directly on $g$~\cite[Ex.~1.2.3]{Nesterov04}. However, unlike gradient 
descent, we provide the additional second-order guarantee 
$\lambda_{\min}(\hess g(y_\epsilon)) \ge 
-\sqrt{\rho\epsilon}$, and thus give a
first-order method with non-asymptotic convergence guarantees to second-order
stationary points at essentially no additional cost over gradient
descent. We remark that concurrent  
works~\cite{AgarwalAlBuHaMa17,CarmonDuHiSi18} give algorithms 
attaining such second-order stationary guarantee with an improved 
first-order complexity scaling roughly as $\epsilon^{-7/4}$.

\subsection{Related work}\label{sec:intro-related}

Despite its non-convexity, the problem~\eqref{eq:problem} can be solved 
to machine precision by means of iterative solution to linear systems of the 
form $(A+\lambda I)x = -b$~\cite{CartisGoTo11}. However, the cost of this 
approach generally grows rapidly with the problem dimension $d$. To 
address this, several researchers propose matrix-free solvers that allow 
trading between solution accuracy and computational cost. 
\citet{Griewank81} and~\citet{WeiserDeEr07} propose variants of the 
conjugate gradient method, \citet{WeiserDeEr07} and~\citet{CartisGoTo11} 
propose Krylov subspace solvers based on the Lanczos method, 
and~\citet{BianconciniLiMoSc15} propose a variant of steepest descent. For 
generic (\ie ``easy case'') problems and assuming infinite precision 
arithmetic, Krylov subspace methods solve~\eqref{eq:problem} exactly in 
$d$ iterations~\cite{CartisGoTo11}, but such guarantees provide limited
insight for high-dimensional problems, where the number of iterations is
typically $\ll d$. Ideally, a matrix-free solver should provide an 
$\eps$-accurate solution to~\eqref{eq:problem} in a number of iterations 
(matrix-vector products)  independent of the problem dimension $d$, 
growing instead as the desired tolerance $\eps$ decreases, as is the case 
for first-order methods in convex optimization. The above-mentioned 
works empirically demonstrate strong performance and scaling to 
high-dimensional problems, but do not provide such dimension-free 
convergence guarantees. Our main result shows that gradient descent 
solves~\eqref{eq:problem} to $\eps$ accuracy in $O(\log(d/\eps)/\eps)$ 
steps, giving a (nearly) dimension-free convergence guarantee. 
Krylov subspace 
methods provide solutions at least as accurate as those of gradient descent 
running the same number of iterations, and therefore our results imply the 
same convergence guarantee for them as well. 

The iterative solvers proposed 
in~\cite{Griewank81,WeiserDeEr07,CartisGoTo11,BianconciniLiMoSc15}
approximate subproblem solutions in the cubic regularization 
scheme~\eqref{eq:nesterov-polyak}. It is therefore interesting to 
understand the total computational cost (in terms of gradient and 
Hessian-vector product evaluations) of finding an 
$\epsilon$-second-order-stationary point for the function $g$ using 
these  approximate solvers. \citet{CartisGoTo11b} show that solving the 
subproblem with a single subspace iteration (known as the Cauchy point) 
is sufficient for the overall method to converge to an $\epsilon$-stationary 
point of $g$ in $O(\epsilon^{-2})$ outer iterations. However, second-order 
stationarity is not guaranteed, and the Nesterov-Polyak rate of 
$O(\epsilon^{-3/2})$ outer iterations is lost. One naturally asks how many 
more iterations of the subproblem solver are needed to restore these 
guarantees. In a follow-up work, \citet{CartisGoTo12} address 
this question by providing conditions on the quality of subproblem 
approximations which suffice to guarantee 
$\epsilon$-second-order-stationarity after $O(\epsilon^{-3/2})$ outer 
iterations. It is unclear how to meet these conditions with a 
matrix-free method, and in Section~\ref{sec:trustregion} we show that solving 
the subproblems with at most $\wt{O}(\epsilon^{-1/2})$ gradient descent 
steps  
guarantees $\epsilon$-second-order-stationarity after 
$O(\epsilon^{-3/2})$ outer 
iterations.

Work on the cubic-regularized problem~\eqref{eq:problem} parallels and 
draws from the literature on the quadratic trust region
problem~\cite{ConnGoTo00, GouldLuRoTo99, GouldRoTh10, ErwayGi09}, 
where one
replaces the regularizer $(\rho/3)\norm{x}^3$ with the constraint
$\norm{x} \le R$. Here too, exact solutions are available but scale poorly
with dimension, and leading matrix-free solvers include the Steihaug-Toint 
truncated conjugate gradient method and GLTR, a Lanczos-based subspace 
method~\cite{GouldLuRoTo99}. 
\citet{TaoAn98} give an
analysis of projected gradient descent with a restart scheme that guarantees 
convergence to the 
global 
minimum; however, the
number of restarts may be proportional to problem dimension,
suggesting potential difficulties for large-scale problems.~\citet{BeckVa18} 
show convergence to the global minimum for a 
family of simple first-order methods that includes projected gradient 
descent. None of these works provides a dimension-free bound on the 
number of iterations required to solve the subproblem to $\eps$ accuracy.

\citet{HazanKo16} address this issue, giving a first-order method 
that
solves the trust-region problem with an accelerated, nearly dimension-free 
rate. 
They find an
$\eps$-suboptimal point for the trust region problem in 
$\wt{O}(1/\sqrt{\eps})$
matrix-vector multiplies 
by reducing the trust-region problem to a sequence of approximate
eigenvector problems. 
\citet{Ho-NguyenKi16} provide a different perspective, showing how a single
eigenvector calculation can be used to reformulate the non-convex quadratic
trust region problem into a convex QCQP, efficiently solvable with first-order 
methods. 

Concurrent to this work, \citet{AgarwalAlBuHaMa17} show the same 
accelerated rate of 
convergence for the cubic problem~\eqref{eq:problem} via
reductions to fast approximate matrix inversion and eigenvector computations.
Their rates of convergence are better than those we achieve when $\eps$ is
large relative to problem conditioning. However, while these works indicate 
that 
solving \eqref{eq:problem} is never harder than approximating the smallest 
eigenvector of $A$, the regime of linear convergence we identify shows that it 
is sometimes much easier.
In work
published during the preparation of this paper, \citet{ZhangShLi17} 
demonstrate that Krylov subspace methods
indeed achieve (accelerated) linear rates of convergence for
trust-region problems, suggesting that such results may be possible
for the cubic-regularized problem~\eqref{eq:problem} as well.

Another related line of work is the study of the behavior of gradient descent 
around saddle-points and its ability to escape them~\cite{GeHuJiYu15, 
LeeSiJoRe16,Levy16}.
A common theme in these works is an ``exponential growth'' mechanism that 
pushes the gradient descent iterates away from critical points with negative 
curvature. 
This mechanism plays a prominent role in our analysis as well, 
highlighting the implications of negative curvature for the dynamics of 
gradient descent.
\section{Preliminaries and basic convergence guarantees}
\label{sec:prel}

We begin by defining some (mostly standard) notation.  
Our problem~\eqref{eq:problem} is to solve
\begin{equation*}
  \minimize_{x\in\R^{d}}~f\left(x\right)\defeq
  \half x^{\T}Ax+b^{\T}x+\frac{\rho}{3} \norm{x}^3,
\end{equation*}
where $\rho>0$, $b\in\R^{d}$ and $A\in \R^{d \times d}$ is a symmetric
(possibly indefinite) matrix, and $\norm{\cdot}$ denotes the Euclidean
norm. 
The eigenvalues of the matrix $A$ are
$\lambda\supind{1} (A) \le \lambda\supind{2} (A) \leq
\cdots\leq\lambda\supind{d}(A)$,
where any of the $\lambda\supind{i}(A)$ may be negative.  We
define the eigengap of $A$ by
$\gap \defeq \lambda\supind{k}(A) - \lambda\supind{1}(A)$ where
$k$ 
is the
first eigenvalue of $A$ strictly larger than $\lambda\supind{1}(A)$. Fix  
$v_1, \ldots, v_d$ to be orthonormal eigenvectors of $A$ such that $A v_i 
= \lambda\supind{i}(A) v_i$, and 
$A = \sum_{i = 1}^d \lambda\supind{i}(A) v_i v_i^T$. 
Importantly, throughout the paper we
work in the eigenbasis of $A$, and for any vector $w \in \R^d$ we let
\begin{equation}\label{eq:superscript-convention}
  w\supind{i} = v_i^T w~\mbox{denote the $i$th
    coordinate of $w$ in the eigenbasis of $A$.} 
\end{equation}

We let $\opnorm{\cdot}$ be the $\ell_2$-operator norm, so $\opnorm{A} =
\max_{u : \norm{u} = 1} \norm{A u}$, and define
\begin{equation*}
  \gamma \defeq -\lambda\supind{1}(A)
  ~~ \mbox{and} ~~
  \beta \defeq \opnorm{A} = 
  \max\{|\lambda\supind{1}(A)|,|\lambda\supind{d}(A)|\},
\end{equation*}
so that the function $f$ is non-convex if and only if $\gamma > 0$. Our 
results also hold when $\beta \ge
\opnorm{A}$ rather than its exact value.  We say a function $g$ is
$L$-\emph{smooth} on a convex set $X$ if $\norm{\grad g(x) - \grad g(y)} \le
L \norm{x - y}$ for all $x, y \in X$; this is equivalent to $\opnorm{\hess
  g(x)} \le L$ for Lebesgue almost every $x \in X$ and is equivalent to the
bound $ | g(x) - g(y) - \nabla g(y)^{\T}(x - y)| \le \frac{L}{2} \norm{x -
  y}^2 $ for $x, y \in X$.

\subsection{Characterization of $f$ and its global minimizers}

Throughout the paper, we let $\s$ denote a solution to 
problem~(\ref{eq:problem}), \ie 
a global minimizer of $f$, and define the matrix
\begin{equation*}
  \As \defeq A+\rho \norm{\s} I,
\end{equation*}
where $I$ is the $d\times d$ identity matrix. 
We have the following characterization for $\s$,
\begin{proposition}[{cf.~\cite[Theorem 3.1]{CartisGoTo11}}] 
	\label{prop:classic-char}
	A solution $\s$ of problem~\eqref{eq:problem} satisfies
	\begin{equation}
	\nabla f(\s) = \As\s+b =0 
	~~ \mbox{and} ~~ \rho \norm{\s} \ge \gamma,
	\label{eqn:optimality}
	\end{equation}
	and $\s$ is unique whenever $\rho \norm{\s} > \gamma$.
\end{proposition}
We may write the gradient and Hessian of $f$ as
\begin{gather*}
  \grad f\left(x\right)
  = \As (x - \s)
  - \rho (\norm{\s} - \norm{x}) x 
  ~ \mbox{and} ~
  \grad^2 f(x) = A + \rho \norm{x} I + \rho \frac{xx^T}{\norm{x}}.
\end{gather*}

\begin{subequations}
  \label{eqn:fs-bounds}
  The globally minimal value of $f$ admits the expression and
  bound
  \begin{equation}
    \label{eq:fs-lower-bound}
    f\left(\s\right)=\frac{1}{2}\s^{\T}As+b^{\T}\s+\frac{\rho\| \s\| 
    ^{3}}{3}=-\frac{1}{2}\s^{\T}\As\s-\frac{\rho\| \s\| 
    ^{3}}{6}\leq-\frac{\rho\| \s\| ^{3}}{6} ,
  \end{equation}
  and, using the fact that $\s^{\T}\As\s=-b^{\T}\s\leq\| b\| \| \s\| $,
  we derive the lower bound
  \begin{equation}
    \label{eq:fs-upper-bound}
    f\left(\s\right)\geq-\frac{1}{2}\| b\| \| \s\| - \frac{\rho \norm{\s}^3}{6}.
  \end{equation}
\end{subequations}
Algebraic manipulation also shows that
\begin{equation}\label{eq:fx-s-expression}
f\left(x\right)=f\left(\s\right)+\frac{1}{2}(x-\s)^T\As(x-\s) 
+\frac{\rho}{6}\left(\| \s\| -\| x\| \right)^{2}\left(\| \s\| +2\| x\| \right),
\end{equation}
which makes it clear that $\s$ is indeed the global minimum,
as both of the $x$-dependent terms are non-negative and minimized at
$x=\s$, and the minimum is unique whenever $\norm{\s} > \gamma/\rho$,
because $\As \succ 0$ in this case.

The global minimizer admits the following equivalent characterization whenever
the vector $b$ is not orthogonal to the eigenspace associated with
$\lambda\supind{1}(A)$.
\begin{proposition}\label{prop:s}
  If $b^{(1)} \neq 0$, $\s$ is the unique solution to the system defined by 
  \[
  \grad f(s) =0~~\mbox{and}~~
  b^{(1)}s^{(1)} \le 0.
  \]
\end{proposition}
\begin{proof}
  Let $\s'$ satisfy $\grad f(\s') = 0$ and $b^{(1)}\s'^{(1)} \le 0$. Focusing on
  the first (eigen)coordinate, we have
    $0 = [\grad f(\s')]^{(1)} = (-\gamma + \rho\|\s'\|)\s'^{(1)} + b^{(1)}$.
  Therefore, $b^{(1)}\neq 0$ implies both $\s'^{(1)} \neq 0$ and
  $-\gamma + \rho\|\s'\| \neq 0$. This strengthens the inequality
  $b^{(1)}\s'^{(1)} \le 0$ to $b^{(1)}\s'^{(1)} <
  0$. Hence
    $-\gamma + \rho\|\s'\| = -b^{(1)}\s'^{(1)}/[\s'^{(1)}]^2 > 0$;
  by Proposition~\ref{prop:classic-char}, if a critical point satisfies 
  $\rho\norm{\s'} > \gamma$ it is
  the unique global minimum.
\end{proof}

The norm of $\s$ plays 
an important role in
our analysis, so we provide a number of bounds on it. 
\begin{subequations}
  \label{eqn:b-bounds-on-s}
  First, observe that $\norm{b} = \norms{\As \s} \ge (-\gamma +\rs) 
  \norm{\s}$. Solving for
  $\|\s\|$  gives the upper bound
  \begin{equation} \label{eq:R-def}
    \| \s\| \leq \frac{\gamma}{2\rho} + \sqrt{\left(\frac{\gamma}{2\rho}\right)^{2}
      +\frac{\| b\| }{\rho}} \leq \frac{\beta}{2\rho}
    +\sqrt{\left(\frac{\beta}{2\rho}\right)^{2}+\frac{\| b\| }{\rho}} \eqdef R
  \end{equation}
  where we recall that $\beta = \opnorm{A}
  \ge |\gamma|$.
  An analogous lower bound on $\norm{\s}$ is available:
  we have $\norm{\s} \ge \gamma / \rho$, and if $b\supind{1} \neq 0$,
  then
  $\|\s\| =\|\As^{-1}b\| \ge |b^{(1)}|/(-\gamma +\rho \|\s\|)$ implies
  \begin{equation}
    \| \s\| \geq\frac{\gamma}{2\rho}+\sqrt{\left(\frac{\gamma}{2\rho}\right)^{2} 
    + \frac{|b^{(1)}|}{\rho}}
    \geq -\frac{\beta}{2\rho}+\sqrt{\left(\frac{\beta}{2\rho}\right)^{2} + 
    \frac{|b^{(1)}|}{\rho}} = R - \frac{\beta}{\rho}.
    \label{eq:s-lower-R-def}
  \end{equation}
\end{subequations}
We can also prove a different lower bound with the similar form
\begin{equation}
  \label{eq:Rc-def}
  \| \s\| \geq R_{c}\defeq
  \frac{-b^{\T}Ab}{2\rho\| b\| ^{2}}+\sqrt{\left(\frac{b^{\T}Ab}{2\rho\| b\| ^{2}}\right)^{2}+\frac{\| b\| }{\rho}}
  \geq
  -\frac{\beta}{2\rho}+\sqrt{\left(\frac{\beta}{2\rho}\right)^{2}+\frac{\| b\| }{\rho}} .
\end{equation}
The quantity $R_c$ is the \emph{Cauchy radius}~\cite{ConnGoTo00}---the 
magnitude of the (global) minimizer of $f$ in the subspace spanned by $b$:
$R_c = \argmin_{\zeta\in\R} f(-\zeta b/\|b\|)$.  To see the claimed lower
bound~\eqref{eq:Rc-def}, set $x_c = -R_c b/\|b\|$ (the \emph{Cauchy 
point}) and note that
$f(x_c) = -(1/2)\|b\|R_c -(\rho/6)R_c^3$. Therefore,
  $0 \le f(x_c) - f(\s) \le \frac{1}{2}\|b\|(\|\s\|-R_c) + \frac{1}{6}\rho(\|\s\|^3 - 
  R_c^3)$,
which implies $\|\s\| \ge R_c$. 

For matrices $A$ with distinct eigenvalues, $f$ may have a single 
suboptimal local minimizer, a single local maximizer and up to $2(d-1)$ 
saddle points~\cite[Section 3]{Griewank81}; see
Figure~\ref{fig:simple-gradients} for an example with $d=2$.

\subsection{Properties and convergence of gradient descent}

\begin{figure}
  \begin{center}
	\includegraphics[height=6cm]{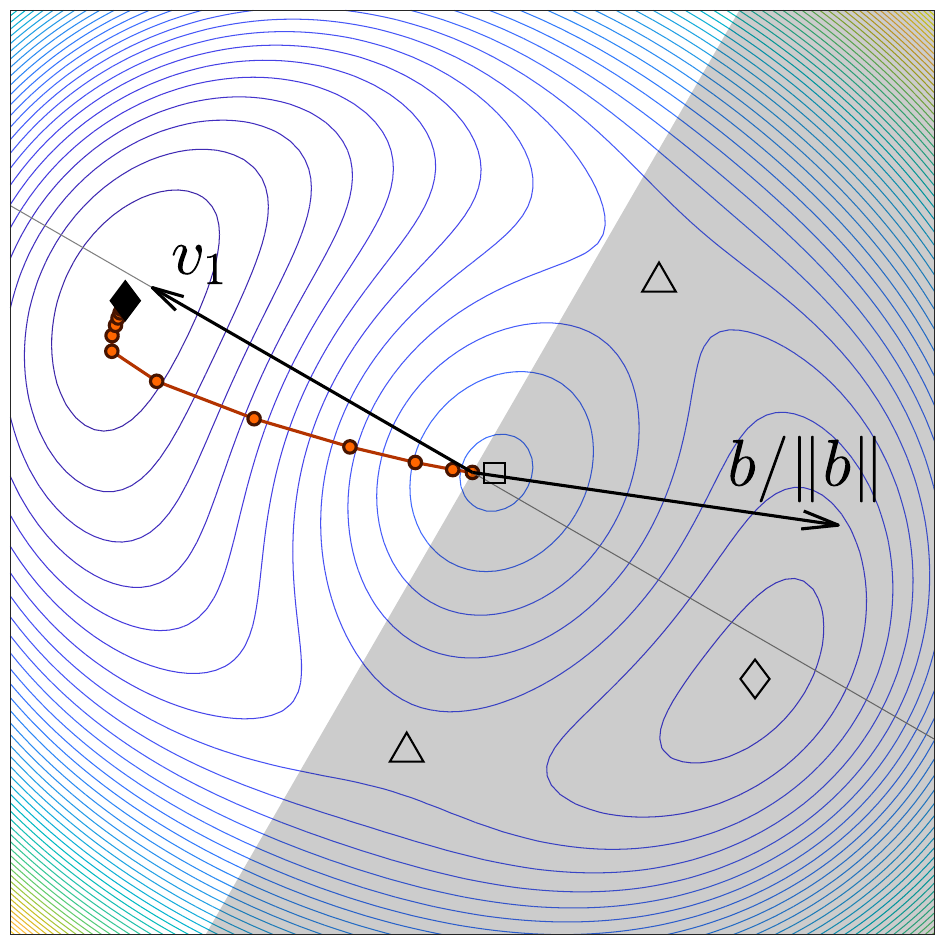}
    \caption{\label{fig:simple-gradients} Contour plot of
    	  a two-dimensional instance of \eqref{eq:problem}, featuring a 
    	  local maximum ($\square$), saddle points ($\triangle$), and
    	  local minima ($\lozenge$). The line of circles indicates the path of 
    	  gradient descent initialized at the origin, and the grey area is the 
    	  half-plane $(v_1^T b)(v^T x) = b\supind{1}x\supind{1} > 0$. Note 
    	  that the global minimum is the only critical point outside this 
    	  half-plane (Proposition~\ref{prop:s}). The gradient descent iterates 
    	  have increasing norm (Lemma~\ref{lem:monotone-weak}), lie 
    	  outside the 
    	  half-plane (Lemma~\ref{lemma:signs}), and converge to $\s$ 
    	  (Proposition~\ref{prop:converge}).    	  
  }
  \end{center}
\end{figure}

The gradient descent method begins at some initialization $x_{0} \in \R^d$
and generates iterates via
\begin{equation}
  x_{t + 1} =
  x_t - \eta \grad f(x_t)
  = (I - \eta A - \rho \eta \norm{x_t} I) x_t - \eta b,
  \label{eq:grad-iter}
\end{equation}
where $\eta$ is a fixed step size.
Recalling the definitions~\eqref{eq:R-def} and~\eqref{eq:Rc-def} of $R$ and $R_c$
as well as $\opnorm{A} = \beta$, throughout our
analysis we make the following assumptions.
\begin{assumption}
  \label{assu:step-size}
  The step size $\eta$ in~\eqref{eq:grad-iter} satisfies
  $0<\eta\leq \frac{1}{4(\beta+\rho R)}$.
\end{assumption}
\begin{assumption}
  \label{assu:init}
  The initialization of~\eqref{eq:grad-iter}
  satisfies $x_{0}=- r \frac{b}{\norm{b}}$,
  with $0\leq r\leq R_{c}$.
\end{assumption}
\noindent
To select a step size $\eta$ satisfying
Assumption~\ref{assu:step-size}, only a rough upper bound on 
$\opnorm{A}$ is necessary. One way to obtain such a bound (with high 
probability) is to apply a few power iterations on $A$. Alternatively, we may 
perform line search, as in Section~\ref{sec:linesearch}.

We begin our treatment of the convergence of gradient descent by establishing
that $\|x_t\|$ is monotonic and bounded (see Appendix~\ref{app:prel} for a
proof).

\begin{restatable}{lemma}{lemMonoWeak}
  \label{lem:monotone-weak}
  Let Assumptions~\ref{assu:step-size} and~\ref{assu:init} hold.  Then
  the iterates~\eqref{eq:grad-iter} of gradient descent
  satisfy $x_t^{\T}\grad f(x_t) \le 0$, the norms
  $\norm{x_t}$ are non-decreasing, and $\norm{x_t} \le R$.
\end{restatable}
\noindent
This lemma is the key to our analysis throughout the paper.
The next lemma shows that 
$x_t$ and $b$ have opposite signs at all coordinates
in the eigenbasis of $A$.
\begin{restatable}{lemma}{lemBSXSign}
  \label{lemma:signs}
  Let Assumptions~\ref{assu:step-size} and~\ref{assu:init} hold.
  For all $t\geq0$ and $i\in\{1,...,d\}$
  \begin{equation*}
    x_{t}^{(i)}b^{(i)}\leq0, ~
    b^{(i)}\s^{(i)}\leq0, ~
    \text{and} ~ x_{t}^{(i)}\s^{(i)}\geq0.
  \end{equation*}
  Consequently, $x_t^{\T}b \le 0$ and $x_{t}^{\T}\s\geq0$ for every $t$, and
  $\s^{\T}b\le 0$.
\end{restatable}
\begin{proof}
  We first show that $x_{t}^{(i)}b^{(i)}\leq0$.
  Writing the gradient descent recursion in the eigenbasis of $A$,
  we have
  \begin{equation}
    x_{t}^{(i)}=\left(1-\eta\lambda^{(i)}(A)-\eta\rho\| x_{t-1}\| \right)x_{t-1}^{(i)}-\eta b^{(i)}.
    \label{eq:coord-rec}
  \end{equation}
  Assumption~\ref{assu:step-size} and Lemma \ref{lem:monotone-weak}
  imply $1-\eta\lambda^{(i)}(A)-\eta\rho\| x_{t-1}\|
  \ge 1 - \eta(\beta + \rho R) > 0$ for all $t, i$.
  Therefore, $x_{t}^{(i)}b^{(i)}\leq0$
  if $x_{0}^{(i)}b^{(i)}\leq0$;
  the initialization in Assumption~\ref{assu:init} guarantees this.
  To show $b^{(i)}\s^{(i)}\leq0$, we use the fact that
  $b=-\As\s$ to write
  \begin{equation*}
    b^{(i)}\s^{(i)}=-\left(\lambda^{(i)}(A)+\rho\| \s\| \right)[\s^{(i)}]^{2}\leq0
  \end{equation*}
  as $\lambda^{(i)}(A)+\rho\| \s\| \geq0$
  for every $i$ by the condition~\eqref{eqn:optimality}
  defining $\s$.

  Multiplying $x_{t}^{(i)}b^{(i)}\leq0$ and
  $b^{(i)}\s^{(i)}\leq0$ yields
  $x_{t}^{(i)}\s^{(i)}[b^{(i)}]^{2}\geq0$.
  The coordinate-wise update~\eqref{eq:coord-rec} and
  Assumption~\ref{assu:init} show that
  $b^{(i)}=0$ implies $x_{t}^{(i)}=0$ for every
  $t$, and therefore $x_{t}^{(i)}\s^{(i)}\geq0$.
\end{proof}

Lemmas~\ref{lem:monotone-weak},
\ref{lemma:signs}, and Proposition~\ref{prop:s} immediately lead to the 
following 
guarantee.

\begin{restatable}{proposition}{propConverge}
  \label{prop:converge}
  Let Assumptions~\ref{assu:step-size} and~\ref{assu:init} hold, and
  assume that $b\supind{1} \neq 0$. Then
  $x_t \to \s$ and $f(x_t)\downarrow f(\s)$ as $t \to \infty$.
\end{restatable}
\begin{proof}
  By Lemma~\ref{lem:monotone-weak}, the iterates satisfy $\|x_t\|\le R$ for all
  $t$. Since $\opnorm{\nabla^2 f(x)} \le \beta + 2 \rho \norm{x}$, the
  function $f$ is $\beta+2\rho R$-smooth on the set
  $\{x \in \R^d : \norm{x} \le R\}$ containing all the iterates
  $x_t$.  Therefore, by the definition of smoothness and the gradient step,
  \begin{equation*}
    f(x_{t+1}) \le f(x_t) - \eta \|\grad f(x_t)\|^2 + \frac{\eta^2}{2}(\beta+2\rho R)\|\grad f(x_t)\|^2 \le f(x_t) - \frac{\eta}{2} \|\grad f(x_t)\|^2 \ ,
  \end{equation*}
  where final inequality used Assumption~\ref{assu:step-size}
  that $\eta \le \frac{1}{4(\beta + \rho R)}$. Consequently, $f(x_t)$ is 
  decreasing and for every $t>0$,
  \begin{equation}\label{eq:square-grad-sum}
    \frac{\eta}{2} \sum_{\tau=0}^{t-1} \|\grad f(x_\tau)\|^2
    \le f(x_0)-f(x_t) \le f(x_0)-f(\s).
  \end{equation}
  
  Let $\s'$ be any limit point of
  the sequence $x_t$ (there must be at least one, as the sequence $x_t$ is 
  bounded). Inequality~\eqref{eq:square-grad-sum} implies 
  $\grad f(x_t) \to 0$ and 
  therefore $\grad f(\s') = 0$ by continuity. By
  Lemma~\ref{lemma:signs},
  $x_{t}^{(1)}b^{(1)}\leq0$ for every $t$, so
  ${\s'}\supind{1} b\supind{1} \le 0$. Proposition 
  \ref{prop:s} thus implies that $\s'$ is the unique global minimizer $\s$.
  We conclude that $\s$
  is the only limit point of the sequence $x_t$.
\end{proof}

\newcommand{\shat}{\hat{x}_\star}

To handle the case $b^{(1)} =0$, let $k\ge1$ be the first index for which 
$b^{(k)}\neq0$ (if no such $k$ exists then $b=0$ and $x_t=0$ for all $t$).
Consider a modified problem instance, with $b,\rho$ unchanged but $A$ 
replaced with $\tilde{A}=\beta\sum_{i=1}^{k-1} v_i 
v_i^T+\lambda\supind{i}(A)v_i v_i^T $, \ie we replace the $k-1$ smallest 
eigenvalues with $\beta\ge\lambda\supind{d}(A)$. Note that gradient 
descent 
produces the same iterates on the modified and original problems. 
 Additionally, note that Lemma~\ref{lem:monotone-weak} and 
 Proposition~\ref{prop:converge} apply to the modified problem, as the 
 inner 
 product between $b$ and the eigenvector of $\tilde{A}$ corresponding to 
 its smallest 
 eigenvalue is non-zero. Applying these results, we have 
 $\norm{x_t}\uparrow\norm{\shat}$, where $\shat$ is the unique solution 
 of the modified problem. Finally, we have 
 $\norm{\shat}\le 
 \norm{\s}$, since $\shat\neq \s$ only if $\rho\norm{\shat} \le 
 \gamma$~\cite[Sec.~6.1]{CartisGoTo11}. Thus, we obtain the following 
 lemma, to which we will refer throughout the sequel.
\begin{lemma}
  \label{lemma:monotone}
  Let Assumptions~\ref{assu:step-size} and~\ref{assu:init} hold.  For all
  $t \ge 0$, the iterates~\eqref{eq:grad-iter} of gradient descent satisfy
  $x_t^{\T}\grad f(x_t) \le 0$, the norms $\norm{x_t}$ are non-decreasing and
  satisfy $\norm{x_t} \le \norm{\s}$, and $f$ is
  ($\beta + 2 \rho \norm{\s}$)-smooth on a ball containing the iterates 
  $x_t$.
\end{lemma}
\noindent
Figure~\ref{fig:simple-gradients} provides a graphical representation of these
results, showing gradient descent's iterates on an instance of
problem~\eqref{eq:problem} exhibiting numerous stationary points.
\section{Non-asymptotic convergence rates}
\label{sec:mainresult}

Proposition~\ref{prop:converge} shows the convergence of gradient
descent for the cubic-regularized (non-convex) quadratic
problem~\eqref{eq:problem}. We now present stronger
non-asymptotic guarantees,
including a randomized scheme solving~\eqref{eq:problem} 
in all cases. We follow with simulations 
illustrating our theoretical results.

\subsection{Theoretical results}
\label{sub:statement}

Our primary result, Theorem~\ref{thm:main-result}, gives a convergence
rate for gradient descent in the case that $b\supind{1} \neq 0$. (Recall 
our convention~\eqref{eq:superscript-convention}, that parenthesized 
superscripts denote components in the eigenbasis of $A$). Further recalling 
that $\gamma = -\lambda\supind{1}(A)$, $\beta = \opnorm{A}$,
$\gap$ is the eigengap of $A$, we define the shorthand
\begin{equation*}
\gammaplus \defeq \max\{\gamma, 0\}~~\mbox{and}~~
\gap' \defeq \min\{\gap, \rs\}.
\end{equation*}
With this notation in hand, we state our result as follows.

\begin{theorem}
  \label{thm:main-result}
  Let Assumptions \ref{assu:step-size} and \ref{assu:init} hold, $b\supind{1} 
  \neq 0$,
  and $\eps > 0$. Then $f(x_t) \le f(\s) + \eps$ for all
    \begin{equation}
  t\geq T_{\eps} \defeq \frac{\tgrow(b^{(1)})+\tconv\left(\eps\right )}{\eta} 
  \begin{cases}
 \frac{1}{\rs -\gamma} & %
  	\frac{1}{\rs -\gamma}  \le \frac{10 \norm{\s}^2}{\eps} 
  	\vspace{5pt}\\ 
   \sqrt{ \frac{10 \norm{\s}^2}{\eps} \cdot \frac{1}{\gap'} } &  %
  	\frac{1}{\gap'} \le \frac{10 \norm{\s}^2}{\eps} \le \frac{1}{\rs -\gamma}  
  	\vspace{5pt}\\ 
   \frac{10 \norm{\s}^2}{\eps}  & %
  \textnormal{otherwise}
  \end{cases}
  \label{eq:main-bound}
  \end{equation}
  where
  \begin{equation*}
    \tgrow(b\supind{1}) = 
    6\log\left(1+\frac{\gammaplus^2}{4\rho|b\supind{1}|}\right)
    ~ \mbox{and} ~
    \tconv(\eps)=6\log\left(\frac{ (\beta + 2\rs)\norm{\s}^2 
   }{\eps}\right).
  \end{equation*}
\end{theorem}
\noindent
See Section~\ref{sec:proof-main-result} for a proof. 

Theorem~\ref{thm:main-result} shows that the rate of convergence changes 
from roughly $O(1/\eps)$ to
$O(\log(1/\eps))$ as $\eps$ decreases, with an intermediate
gap-dependent rate of $O(1/\sqrt{\eps})$. The terms
$\tgrow$ and $\tconv$ correspond to a period ($\tgrow$) in which
$\norm{x_t}$ grows exponentially in $t$ until reaching the basin of
attraction to the global minimum and a period ($\tconv$) of linear
convergence to $\s$. Exponential growth occurs only in non-convex problem 
instances, as $\tgrow = 0$ when the problem is convex.

The dependence of our result on $|b\supind{1}|$ (the magnitude of $b$ 
in along the direction of the smallest eigenvector of $A$) is unavoidable: 
if $b\supind{1} = 0$, then gradient descent always remains in a subspace
orthogonal to the smallest eigenvector of $A$, while $\s\supind{1}$ 
might
be non-zero; this is the ``hard case'' of non-convex quadratic
problems~\cite{ConnGoTo00, CartisGoTo11}. 
We use a small random perturbation to guarantee
$|b\supind{1}| \neq 0$ except with negligible probability, which yields
the following high probability guarantee, whose proof we provide in
Section~\ref{sec:proof-pert-main-result}.

\begin{theorem}
  \label{thm:pert-main-result}
  Let Assumptions~\ref{assu:step-size} and \ref{assu:init} hold, $\eps,\delta > 
  0$, and let $q$ be uniformly distributed on the unit sphere in $\R^d$.
  Let $\tilde{x}_t$ be generated by the gradient descent
  iteration~\eqref{eq:grad-iter} with $\tilde{b}=b+\sigma q$ replacing $b$, 
  where
  \begin{equation*}
    \sigma = \frac{\rho\eps^2}{200(\beta+2\rs)^2 \nss}\cdot 
    \sigbar~\mbox{with}~\sigbar\le1.
  \end{equation*}
  Then with probability at least $1 - \delta$, we have 
  $f(\tilde{x}_t) \le f(\s) + (1+\sigbar)\eps$ 
  for all
      \begin{equation}
  t\geq T_{\eps} \defeq 
  \frac{\tiltgrow(d,\delta,\sigbar)+\tiltconv\left(\eps\right 
  )}{(1+\sqrt{\sigbar})^{-2}\eta} 
  \begin{cases}
  \frac{1}{\rs -\gamma} & %
  \frac{1}{\rs -\gamma}  \le \frac{10 \norm{\s}^2}{\eps} 
  \vspace{5pt}\\ 
  \sqrt{ \frac{10 \norm{\s}^2}{\eps} \cdot \frac{1}{\gap'} } &  %
  \frac{1}{\gap'} \le \frac{10 \norm{\s}^2}{\eps} \le \frac{1-2\sqrt{\sigbar}/3}{\rs 
  -\gamma}  
  \vspace{5pt}\\ 
  \frac{10 \norm{\s}^2}{\eps}  & %
  \textnormal{otherwise}
  \end{cases}
  \label{eqn:always-true-convergence}
  \end{equation}
  where
  \begin{equation*}
    \tiltgrow(d,\delta,\sigbar)
    \defeq 
    6\log\left(1+\I_{\{\gamma>0\}}\frac{50\sqrt{d}}{\sigbar\delta}\right)
    ~ \mbox{and} ~
    \tiltconv(\eps) \defeq
    20\log\left(\frac{ (\beta + 2\rs)\norm{\s}^2 
    }{\eps}\right)\!\!.
  \end{equation*}
\end{theorem}

To facilitate later discussion, we define
$\Ls \defeq \beta + 2\rs$; then $f$ is $\Ls$-smooth on the Euclidean ball of
radius $\norm{\s}$. The bound~\eqref{eq:s-lower-R-def} implies $\rho R 
\le \beta + \rs$, and therefore the step size choice $\eta = \frac{1}{4(\beta 
+ \rho R)}$ satisfies $\frac{1}{\eta} \le 8\beta +4\rs \le 8\Ls$. Combining 
this upper bound with Theorem~\ref{thm:pert-main-result}, we  have the 
following corollary.
\begin{corollary}\label{cor:main-result}
  Let the conditions of Theorem~\ref{thm:pert-main-result} hold,
  $\eta = \frac{1}{4(\beta + \rho R)}$ and $\sigbar = 1$.
  Then with probability at least $1 - \delta$,
  we have $f(\tilde{x}_t) \le f(\s) + \eps$ for all
  \begin{equation*}
    t \ge \tilde{T}_{\eps} =
    O(1) \cdot
      \min \left\{
        \frac{\Ls}{\rs - \gamma},
        \frac{\Ls \nss}{\eps} \right\}
      \log\left[
        \left(1 + \I_{\{\gamma>0\}}\frac{ d}{\delta}\right)
        \frac{\Ls \nss}{ \eps} \right].
  \end{equation*}
\end{corollary}

We conclude the presentation of our main results with a few brief remarks.

\begin{enumerate}[(i)]
\item Corollary~\ref{cor:main-result} highlights parallels between our
  guarantees and those for gradient descent on smooth convex
  functions~\cite{Nesterov04}. In our case, $\Ls/(\rs-\gamma) \ge 1$
  is a condition number, while $\Ls$ and $\norm{\s}$ bound the
  smoothness of $f$ and iterate radius $\sup_t\norm{x_t}$, respectively.
  We defer further comparison to
  Section~\ref{sec:discussion}.
\item We readily obtain relative accuracy guarantees by using the
  bound~\eqref{eq:fs-lower-bound}; setting
  $\eps = \rho\norm{\s}^3\eps'/12$, we have
  $f(\tilde{x}_t)-f(\s) \le -\eps'f(\s) = \eps'(f(0)-f(\s))$,
  or $f(\tilde{x}_t) \le (1 - \eps') f(\s)$, for any
  $t\ge \tilde{T}_\eps$, where $\tilde{T}_\eps$ is defined
  in~\eqref{eqn:always-true-convergence}.
\item Evaluating $\tilde{T}_\eps$ for given $A$, $b$ and $\rho$ is not
  straightforward, as $\norm{\s}$ is generally unknown. Using
  $\norm{\s}\le R$ gives an easily computable upper
  bound on $\tilde{T}_\eps$, and in Section~\ref{sec:trustregion}, we
  demonstrate how to apply our results when $\norm{\s}$ is unknown.
\end{enumerate}
\subsection{Illustration of results}\label{sec:experiments}

\begin{figure}
  \begin{center}
    \includegraphics[width=0.9\columnwidth]{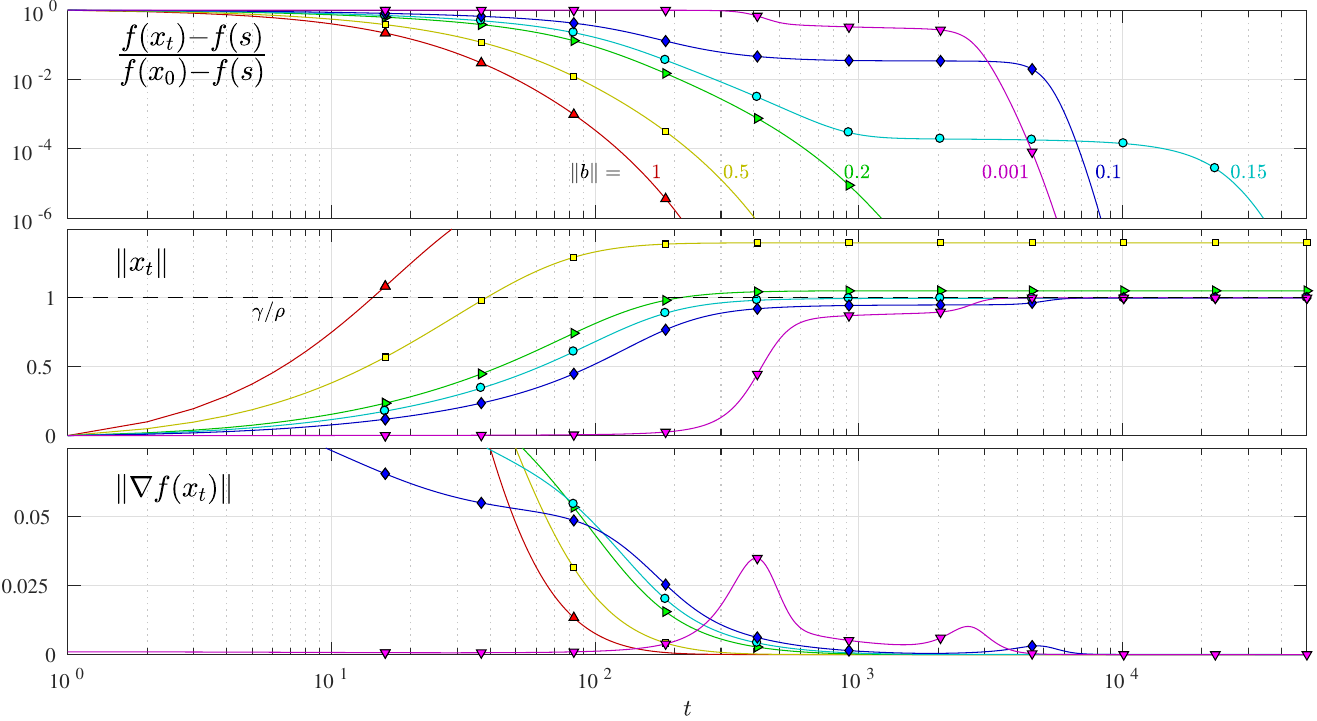}
    \caption{\label{fig:trajB} Trajectories of gradient descent with
      $\lambda^{(1)}(A) = -\gamma = -0.2$ and
      $\lambda^{(2)}(A),...,\lambda^{(d)}(A)$ equally spaced between $-0.18$
      and $\beta=1$, and different vectors $b$ proportional to $[0.01, 1, 1, 
      1, \ldots]$ in the eigenbasis of $A$.
      The rest of the parameters
      are $d=10^3$, $\eta = 0.1$, $\rho = 0.2$ and $x_0=0$. }
  \end{center}
\end{figure}

We present two experiments that investigate the behavior of gradient 
descent on problem~\eqref{eq:problem}.  For the first experiment, we examine 
the behavior of gradient descent on single problem instances, looking at 
convergence behavior as we vary the vector $b$ (to effect conditioning of the 
problem) by scaling its norm $\norm{b}$. The selected norm values 
$\norm{b}\in\{1, 0.5, 0.2, 0.15, 0.1, 0.001\}$ correspond to condition 
numbers $(\beta+\rs)/(-\gamma + \rs)\in \{7.6, 16, 120, 
5.5\cdot10^{3}, 2.9\cdot10^{4}, 3.8\cdot10^{6}\}$; the problem 
conditioning becomes worse as $\norm{b}$ decreases.   
Figure~\ref{fig:trajB} 
summarizes our results and describes the settings of the other 
parameters in the experiment.

The plots show two behaviors of gradient descent.  The problem is 
well-conditioned when $\|b\|\ge 0.2$, and in these cases gradient descent 
behaves as though the problem was strongly convex, with $x_t$ converging 
linearly to $\s$. For $\|b\| \le 0.15$ the problem 
becomes ill-conditioned and gradient descent stalls around saddle points. 
Indeed, 
the third plot of Figure~\ref{fig:trajB} shows that for the ill-conditioned 
problems, we have $\norm{\nabla f(x_t)}$ increasing over some iterations, 
which does not occur in convex quadratic 
problems.  The length of the stall does not depend only on the condition 
number; for $\norm{b} = 10^{-3}$ the stall is shorter than for 
$\norm{b} \in \{0.1, 0.15\}$. Instead, it appears to depend on the norm of 
the 
saddle point which causes it, which we observe from the value of 
$\norm{x_t}$ at the time of the stall; we see that the closer the norm is to 
$\gamma/\rho$, the longer the stall takes. This is explained by observing 
that $\hess f(x) \succeq (\rho\norm{x}-\gamma)I$, which means that every 
saddle point with norm close to $\gamma/\rho$ must have only 
small negative curvature, and therefore harder to escape (see also 
Lemma~\ref{lem:growth} in the sequel). Fortunately, as we see in 
Fig.~\ref{fig:trajB}, saddle points with large norm have near-optimal objective 
value---this is the intuition behind our proof of the sub-linear convergence 
rates. 

\begin{figure}
  \begin{center}
    \heavyplot{\includegraphics[width=0.9\columnwidth]%
      {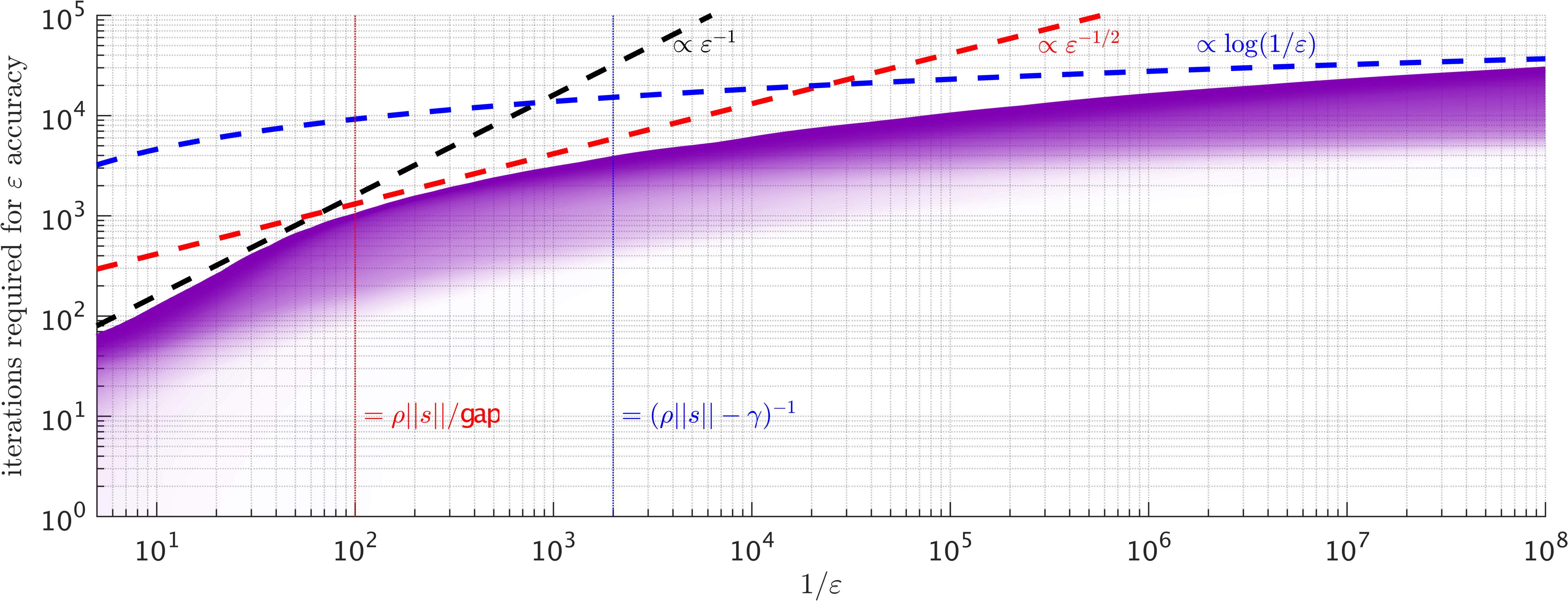}}
    \caption{\label{fig:ensemble} The purple curve is shaded according to the
      cdf of the number of iterations required to reach relative accuracy
      $\eps$, computed over $2{,}500$ random problem instances each with
      $d=10^4$, $\beta=\rho=1$, $\gamma=0.5$, $\gap = 5\cdot 10^{-3}$ and
      $\rs-\gamma = 5\cdot 10^{-4}$. We use $x_0=-R_c b/\|b\|$ and
      $\eta = 0.25$.  The black, red and blue dashed curves indicate the three
      convergence regimes Theorem~\ref{thm:main-result} identifies.}
  \end{center}
\end{figure}

In our second experiment, we test our rate guarantees by considering the
performance of gradient descent over an ensemble of random instances.  We
generate random instances with a fixed value of $\gamma$, $\beta$, $\rho$,
$\norm{\s}$ and $\gap$ as follows. We set
$A=\diag ([-\gamma; -\gamma+\gap; u])$ with $u$ uniformly random in
$[-\gamma+\gap,\beta]^{d - 2}$. We draw $\check{\s} = (A+\rs)^{-\zeta}\nu$,
where $\nu\sim\mathcal{N}(0;I)$ and $\log_2 \zeta$ is uniform on
$[-1,1]$. We then set $\s=(\norm{\s}/\norm{\check{\s}})\check{\s}$ and
$b = -(A+\rs)\s$, so that $\s$ is the global minimizer of problem instance
$(A,b,\rho)$. The choice of $\zeta$ ensures we observe large variety in the 
values of $\norm{x_t}$ at which gradient descent stalls,  
allowing us to find difficult 
instances for each value of $\eps$. In Figure
\ref{fig:ensemble} we depict the cumulative distribution of the number of
iterations required to find an $\eps$-relatively-accurate solution versus
$1/\eps$. The slopes in the plot agree with our upper
bounds, suggesting the sharpness of our theoretical results.
\section{Proofs of main results}
\label{sec:proofs}

In this section, we provide proofs of our main results,
Theorems~\ref{thm:main-result} and~\ref{thm:pert-main-result}.  A number of
the steps involve technical lemmas whose proofs we defer to
Appendix~\ref{app:proof}. In all lemma statements, we tacitly let
Assumptions~\ref{assu:step-size} and \ref{assu:init} hold, as in the main 
theorem
statements. Without loss of generality, we assume
$\eps \le \half \beta \norm{\s}^2 + \rho \norm{\s}^3$, as
$f$ is $\beta + 2 \rho \norm{\s}$ smooth on the set $\{x : \norm{x} \le 
\norm{\s}\}$ and therefore $f(x_0) \le f(\s) + \eps$ for any $\eps \ge \half \beta 
\norm{\s}^2 + \rho \norm{\s}^3$.

\subsection{Proof of Theorem~\ref{thm:main-result}}
\label{sec:proof-main-result}

We divide the proof of Theorem~\ref{thm:main-result} into two main steps: 
in Section~\ref{sec:linear-convergence-part} we prove the first case in the 
bound~\eqref{eq:main-bound} (linear convergence), and in 
Section~\ref{sec:epsilon-convergence-part} we prove the last two cases 
in~\eqref{eq:main-bound} (sublinear convergence).

\subsubsection{Linear convergence and exponential growth}
\label{sec:linear-convergence-part}

We first prove that $f(x_t) \le f(\s) + \eps$ for
$t \ge \frac{1}{\eta (\rho \norm{\s} -
	\gamma)}(\tgrow(b\supind{1}) + \tconv(\eps))$. We begin with two lemmas 
	that provide regimes in which
$x_t$ converges to the solution $\s$ linearly.

\begin{restatable}{lemma}{lemLinConv}
  \label{lem:lin-converge}
  For each $t > 0$, we have
  \begin{equation*}
    \norm{x_t - \s}^2 \le 
    \left(1-\eta \left[\rho\norm{x_t}-\left(\gamma - 
    \frac{\rs-\gamma}{2}\right)\right]\right)
    \norm{x_{t-1}-\s}^2
  \end{equation*}
\end{restatable}
\noindent
See Appendix~\ref{sec:proof-lin-converge} for a proof of this lemma.

For non-convex problem instances (those with $\gamma > 0$), the above
recursion is a contraction (implying linear convergence of $x_t$ to $\s$) only 
when
$\rho\norm{x_t}$ is larger than $\gamma - \half(\rs-\gamma)$. 
Using the fact that $\|x_t\|$
is non-decreasing (Lemma~\ref{lemma:monotone}),
Lemma~\ref{lem:lin-converge} immediately implies the following result.
\begin{lemma}
  \label{lem:really-lin-converge}
  If $\rho\| x_{t}\| \geq\gamma-\frac{1}{2}\left(\rs-\gamma\right)+\mu$
  for some $t\geq0$, then for all $\tau\geq0$,
  \begin{equation*}
    \| x_{t+\tau}-\s\|^2
    \leq\left(1-\eta\mu\right)^{\tau}\| 
    x_{t}-\s\|^2\leq2\norm{\s}^2e^{-\eta\mu\tau}.
  \end{equation*}
\end{lemma}
\begin{proof}
  Lemma~\ref{lem:lin-converge} implies that
  $\norm{x_{t + \tau} - \s}^2 \le (1 - \eta \mu) \norm{x_{t + \tau - 1} -
    \s}^2$
  for all $\tau > 1$. Using that
  $\norm{x_t - \s}^2 \le \norm{x_t}^2 + \norm{\s}^2 \le 2\norm{\s}^2$
  by Lemma~\ref{lemma:signs}, Lemma~\ref{lemma:monotone}, and
  $1 + \alpha \le e^\alpha$ for all $\alpha$ gives the result.
\end{proof}

It remains to understand whether the gradient descent iterations satisfy the 
condition $\rho\| x_{t}\| \ge 
\gamma-\frac{1}{2}\left(\rs-\gamma\right)+\mu$. Fortunately, as long 
as $\rho\| x_{t}\| $
is below $\gamma-\smallval$, $|x_{t}^{\left(1\right)}|$ grows faster than 
$(1 + 
\eta \smallval)^t$:
\begin{restatable}{lemma}{lemGrowth}
  \label{lem:growth}
  Let $\smallval>0$. Then $\rho\| x_{t}\| \geq\gamma-\smallval$ for all
  $t\geq\frac{2}{\eta\smallval}\log(1+\frac{\gammaplus^2}{4\rho|b^{\left(1\right)}|})$.
\end{restatable}
\noindent
See Appendix~\ref{sec:proof-growth} for a proof of this lemma.

We now combine the lemmas to give the linear convergence regime of
Theorem~\ref{thm:main-result}.  Applying Lemma~\ref{lem:growth} with
$\smallval=\frac{1}{3}(\rho\norm{\s} -\gamma)$ yields 
$\rho\norm{x_{t}} \geq\gamma-\frac{1}{3}\left(\rho\norm{\s} -\gamma\right)$ 
for
\begin{equation*}
  t\geq T_{1}\triangleq\frac{6}{\eta\left(\rho\norm{\s}
      -\gamma\right)}\log\left(1+\frac{\gammaplus^2}{4\rho|b^{\left(1\right)}|}\right)
  \IndNC
  = \frac{1}{\eta (\rho \norm{\s} - \gamma)} \tgrow(b\supind{1}).
\end{equation*}
Therefore, by Lemma~\ref{lem:really-lin-converge} with $\mu = 
\half(\rs-\gamma)-\smallval=
\frac{1}{6}(\rs-\gamma)$, for any $t$ we have
\begin{equation}
  \label{eqn:eventual-linear-convergence}
  \norm{x_{T_{1}+t} - \s}^{2}
  \leq2\norm{\s}^{2} \exp\Big(-\frac{1}{6}\eta\left(\rho\norm{\s} 
  -\gamma\right)t\Big).
\end{equation}
As a consequence, for all $t \ge 0$ we may use the
$(\beta + 2 \rho \norm{\s})$-smoothness of $f$ and the fact that
$\norm{x_t}\le \norm{\s}$ (by Lemma~\ref{lemma:monotone}) to
obtain
\begin{equation*}
  f\left(x_{t}\right)-f\left(\s\right)\leq\frac{\beta+2\rho\norm{\s}
  }{2}\norm{x_{t}-\s}^{2}
  \le (\beta + 2 \rho \norm{\s}) \norm{\s}^2
  e^{-\frac{1}{6} \eta (\rho \norm{\s} - \gamma) (t - T_1)}
\end{equation*}
where we have used that $\nabla f(\s) = 0$ and the
bound~\eqref{eqn:eventual-linear-convergence}.  Therefore, if we set
\begin{equation*}
  T_2 \defeq
  \frac{6}{\eta\left(\rho\norm{\s} -\gamma\right)}
  \log \frac{(\beta + 2 \rho \norm{\s}) \norm{\s}^2}{\eps}
  = \frac{1}{\eta (\rho \norm{\s} - \gamma)} \tconv(\eps),
\end{equation*}
then $t \ge T_1 + T_2 = \frac{1}{\eta (\rho \norm{\s} - 
\gamma)}(\tgrow(b\supind{1}) + 
\tconv(\eps))$ implies
  $f(x_t) - f(\s) \le \eps$.

\subsubsection{Sublinear convergence and convergence in subspaces}
\label{sec:epsilon-convergence-part}

We now turn to the sublinear convergence regime in
Theorem~\ref{thm:main-result},
which applies when the quantity $\rs-\gamma$ is sufficiently small.
\begin{equation}
  \label{eqn:assumption-bad-condition}
  \rho \norm{\s} - \gamma \le \frac{\eps}{10 \norm{\s}^2}~.
\end{equation}
Note that if~\eqref{eqn:assumption-bad-condition} fails to hold,
then~\eqref{eq:main-bound} is dominated by the $(\rs-\gamma)^{-1}$  term. 
Therefore, to complete the proof of Theorem~\ref{thm:main-result} it suffices 
to show that if~\eqref{eqn:assumption-bad-condition} holds, then $f(x_t) \le 
f(\s) + \eps$ whenever
\begin{equation}
  \label{eqn:sublinear-main}
  t \ge T^{\rm sub}_\eps \defeq \frac{\tgrow(b\supind{1}) + \tconv(\eps)}{\eta}
  \min\left\{ \frac{10 \norm{\s}^2}{\eps},
    \sqrt{\frac{10 \norm{\s}^2}{(\min\{\gap , \rs\}) \eps}} \right\}.
\end{equation}

Roughly, our proof of the result~\eqref{eqn:sublinear-main} proceeds as
follows: when $\rs - \gamma$ is small, the function $f$ is very smooth along
eigenvectors with eigenvalues close to $-\gamma = \lambda\supind{1}(A)$. It is
therefore sufficient to show convergence in the complementary subspace, which
occurs at a linear rate. Appropriately choosing the gap between the
eigenvalues in the complementary subspace and $\lambda\supind{1}(A)$ to trade
between convergence rate and function smoothness yields the
rates~\eqref{eqn:sublinear-main}.

The following analogs of Lemmas~\ref{lem:lin-converge}
and~\ref{lem:really-lin-converge}
establish subspace convergence.
\begin{restatable}{lemma}{lemProjLinConverge}
  \label{lem:proj-lin-converge}
  Let $\Pi$ be any projection matrix satisfying $\Pi A = A \Pi$ for which
  $\Pi \As\succeq\smallval\Pi$ for some $\smallval>0$.
  For all $t>0$,
  \begin{align*}
    &\norms{ \Pi \As^{1/2}(x_{t}-\s)}^2
     \leq\left(1-\eta\smallval\right)
      \norms{ \Pi \As^{1/2}(x_{t-1}-\s)}^2\\
    & \qquad+\sqrt{8}\eta\rho\left(\norm{ \s} -\norm{ x_{t-1}}
      \right)\left[\rho\left(\norm{ \s} -\norm{ x_{t-1}} \right)\norm{
      x_{t-1}}^2+\opnorm{ (I -\Pi)\As} \norm{ \s}^2\right].
  \end{align*}
\end{restatable}
\noindent
See Appendix~\ref{sec:proof-proj-lin-converge} for a proof. Letting
$\Pi_\smallval = \sum_{i : \lambda\supind{i} \ge \smallval + 
\lambda\supind{1}}
v_i v_i^T$
be the projection matrix onto the span of eigenvectors of $A$
with eigenvalues at least $\lambda\supind{1}(A) + \smallval$,
we obtain the following consequence of 
Lemma~\ref{lem:proj-lin-converge},
whose proof we provide in
Appendix~\ref{sec:proof-eventually-proj-lin-converge}.
\begin{restatable}{lemma}{lemEventuallyProjLinConverge}
  \label{lemma:eventually-proj-lin-converge}
  Let $t\geq0$, $\smallval \ge 0$, and
  define $\bar{\smallval}=\max\{\smallval, \gap\}$.
  If $\rs\le\gamma+\sqrt{\smallval\bar{\smallval}}$
  and $\rho \norm{x_t} 
  \geq\gamma-\frac{1}{3}\sqrt{\smallval\bar{\smallval}}$, then
  for any $\tau \ge 0$,
  \begin{align*}
    \norm{ \Pi_{\smallval}\As^{1/2}\left(x_{t+\tau}-\s\right)}^2
    & \leq\left(1-\eta\bar{\smallval}\right)^{\tau}\norm{
      \Pi_{\smallval}\As^{1/2}\left(x_{t}-\s\right)}^2+13\norm{ 
      \s}^2\smallval \\
    & \leq2\left(\beta+\rho\norm{ \s} \right)\norm{ 
    \s}^2e^{-\eta\bar{\smallval}\tau}+13\norm{ \s}^2\smallval.
  \end{align*}
\end{restatable}

\newcommand{\Tsublinone}{T_1^{\rm sub}}
\newcommand{\Tsublintwo}{T_2^{\rm sub}}
\newcommand{\Tsublin}{T^{\rm sub}}

We use these lemmas to prove the desired bound~\eqref{eqn:sublinear-main}
by appropriate separation of the eigenspaces over which we guarantee
convergence.
To that end, we define
\begin{equation}
  \label{eqn:delta-defs}
  \smallval \defeq \frac{\eps}{10 \norm{\s}^2} ~~ ,
  ~~
  \bar{\smallval}
  \defeq \max\{\smallval, \gap\}
  ~~ \mbox{and} ~~
  \bar{\smallval}' \defeq \max\{ {\smallval},  ( \min\{\gap , \rs\} ) \le 
  \bar{\smallval}\} ,
\end{equation}
and note that the definition of $\gap$
immediately implies $\Pi_{\smallval}=\Pi_{\bar{\smallval}}$.
The growth guranteed by Lemma~\ref{lem:growth} shows that
$\rho\norm{x_{t}} \geq\gamma-\frac{1}{3}\sqrt{\smallval\bar{\smallval}'}$ 
for every
\begin{equation*}
  t \ge \Tsublinone \defeq
  \frac{6}{\eta\sqrt{\smallval\bar{\smallval}'}}\log\left(1+\frac{\gammaplus^2}{
      4\rho|b^{\left(1\right)}|}\right)
  = \frac{1}{\eta \sqrt{\smallval \bar{\smallval}'}} \tgrow(b\supind{1}).
\end{equation*}
Additionally, for $t \ge \Tsublinone$ we have
$\rho\norm{x_t} \geq\gamma-\frac{1}{3} \sqrt{\smallval\bar{\smallval}'}
\geq\gamma-\frac{1}{3}\sqrt{\smallval\bar{\smallval}}$ because
$\bar{\smallval} \ge \bar{\smallval}'$. Thus,
using that $\smallval, \bar{\smallval}' \le \bar{\smallval}$
and that $(\beta + 2 \rs \norm{\s}^2 / \eps \ge 2$ as in
the beginning of Sec.~\ref{sec:proofs}, we may define
\begin{equation*}
  \Tsublintwo \defeq
  \frac{1}{\eta\bar{\smallval}}
  \log \frac{2 (\beta + \rho \norm{\s})}{\smallval}
  \le \frac{1}{\eta \sqrt{\smallval \bar{\smallval}'}}
  \log\left(\left[ \frac{ (\beta + 2\rs)\norm{\s}^2}{\eps}\right]^6 \right)
  = \frac{\tconv(\eps)}{\eta \sqrt{\smallval \bar{\smallval}'}}.
\end{equation*}
Thus $2(\beta + \rs)\norm{\s}^2 e^{-\eta \bar{\smallval} t} \le 
\norm{\s}^2 
\smallval$ for every $t \ge \Tsublintwo$, and
by Lemma~\ref{lemma:eventually-proj-lin-converge} we have 
\begin{equation}
  \norm{\Pi_{\smallval}\As^{1/2}(x_t - \s)}^{2}\leq
  \norm{\s}^2 \smallval + 13\norm{\s}^2 \smallval = 
  14\norm{\s}^{2} \smallval,
  \label{eqn:proj-bound}
\end{equation}
for every $t \ge \Tsublin = \Tsublinone+\Tsublintwo$.

We now translate the guarantee~\eqref{eqn:proj-bound} on the distance from $x_t$ to
$\s$ in the subspace of ``large'' eigenvectors of $A$ to
a guarantee on the solution quality $f(x_t)$. Using
the expression~\eqref{eq:fx-s-expression} for $f(x)$, the
orthogonality of $I - \Pi_\smallval$ and $\Pi_\smallval$ and 
$\norm{x_t}\le\norm{\s}$, we have
\begin{equation*}
  f(x_t) \le %
  f(\s)+ \half \norm{(I - \Pi_{\smallval})\As^{\frac{1}{2}}(x_t - \s)}^{2}
    + \half\norm{\Pi_{\smallval}\As^{\frac{1}{2}}(x_t - \s)}^{2}+
    \frac{\rho\norm{\s} }{2}(\norm{\s} -\norm{x_t} )^{2}.
\end{equation*}
Now we note that
\begin{equation}\label{eq:subpsace-opnorm-bound}
  \opnorm{(I - \Pi_\smallval)\As}
  = \max_{i : \lambda\supind{i} < \lambda\supind{1} + \smallval}
  |\lambda\supind{i} + \rho \norm{\s}|
  \le -\gamma + \smallval + \rho \norm{\s}
  \le 2 \smallval,
\end{equation}
where we have used our assumption~\eqref{eqn:assumption-bad-condition}
that $\rho \norm{\s} - \gamma \le \frac{\eps}{10 \norm{\s}^2} = \smallval$.
Using this gives
\begin{equation}\label{eq:sub-error}
  f(x_t) \le f(\s) + \smallval \norm{x_t - \s}^2
  + 7 \norm{\s}^2 \smallval
  + \frac{\rho \norm{\s}}{2} (\norm{\s} - \norm{x_t})^2,
\end{equation}
where we use inequality~\eqref{eqn:proj-bound}.  Because
$\rho \norm{x_t} \ge \gamma - \frac{1}{3} \sqrt{\smallval \bar{\smallval}'}$
for $t \ge \Tsublinone$, we obtain
\begin{equation*}
  0 \le \rho (\norm{\s} - \norm{x_t})  \le
  \rho \norm{\s} - \gamma - (\rho \norm{x_t} - \gamma)
  \le \frac{4}{3} \sqrt{\smallval \bar{\smallval}'}.
\end{equation*}
The above inequality provides an upper bound on $(\norm{\s}-\norm{x_t})^2$. 
Alternatively, we may bound $(\norm{\s}-\norm{x_t})^2\le 
\norm{\s}^2$ using $\norm{x_t}\le\norm{\s}$ 
(Lemma~\ref{lemma:monotone}). Therefore
\begin{equation}\label{eq:norm-diff-bound}
  \frac{\rho \norm{\s}}{2} (\norm{\s} - \norm{x_t})^2 \le \norm{\s}^2 
  \min\left\{\frac{\rs}{2}, 
    \frac{16 \bar{\smallval}'\smallval}{18\rs}\right\}\le \norm{\s}^2 \smallval,
\end{equation}
where the final inequality follows as $\bar{\smallval}' \le 
\max\{\smallval,  
\rs\}$.
Substituting the bound~\eqref{eq:norm-diff-bound} into 
\eqref{eq:sub-error} with 
$\norm{\s - x_t}^2 \le 2 \norm{\s}^2$ (by Lemma~\ref{lemma:signs}),
we find
\begin{equation*}
  f(x_t)\le f(\s) + 9 \norm{\s}^2 \smallval
  \le  f(x_t) + \eps,
\end{equation*}
where we substitute
$\smallval = \frac{\eps}{10 \norm{\s}^2}$.
Summarizing, if 
$\rho \norm{\s} - \gamma \le \smallval = \frac{\eps}{10 \norm{\s}^2}$,
the point $x_t$ is $\eps$-suboptimal for problem~\eqref{eq:problem} whenever
  $t \ge \frac{\tgrow(b\supind{1}) + 
  \tconv(\eps)}{\eta\sqrt{\smallval\bar{\smallval}'}} 
  \ge \Tsublinone + \Tsublintwo$,
where
\[\sqrt{\smallval\bar{\smallval}'} = \max\left\{\frac{\eps}{10\norm{\s}^2}, 
\sqrt{\frac{\eps}{10\norm{\s}^2}\min\{\gap , \rs\}}\right\}.
\]

\newcommand{\stilde}{\tilde{x}_\star}

\subsection{Proof of Theorem~\ref{thm:pert-main-result}}
\label{sec:proof-pert-main-result}

Theorem \ref{thm:pert-main-result} follows from three basic observations about
the effect of adding a small uniform perturbation to $b$, which we summarize
in the following lemma (see Section~\ref{sec:proof-rand-prop} for a proof).

\begin{restatable}{lemma}{lemPertProp}
  \label{lem:rand-prop}
  Set $\tilde{b}=b+\sigma q$, where $q$ is uniform on the unit
  sphere in $\mathbb{R}^{d}$ and $\sigma>0$.  Let
  $\tilde{f}\left(x\right)=\frac{1}{2}x^{\T}Ax+\tilde{b}^{\T}x+\frac{1}{3}\rho\norm{
    x}^3$ and let $\stilde$ be a global minimizer of
  $\tilde{f}$. Then, for any $\delta > 0$
  \begin{enumerate}[(i)]
  \item For $d > 2$, $\P( |\tilde{b}\supind{1}|
    \le \sqrt{\pi}\sigma\delta/\sqrt{2 d}) \leq\delta$
  \item $| f(x)-\tilde{f}(x) | \le \sigma\norm{x}$
    for all $x\in\mathbb{R}^{d}$
  \item $\big| \norm{\s}- \norm{\stilde} \big| \leq \sqrt{2\sigma/\rho}$.
  \end{enumerate}
\end{restatable}

With Lemma~\ref{lem:rand-prop} in hand, our proof proceeds in three parts: in
the first two, we provide bounds on the iteration complexity of each of the
modes of convergence that Theorem~\ref{thm:main-result} exhibits
in the perturbed problem with vector $\tilde{b}$. The final part shows
that the quality of the (approximate) solutions $\tilde{x}_t$ and $\stilde$
is not much worse than $\s$. 

Let $\tilde{f}, \tilde{b}$ and $\stilde$ be as defined in 
Lemma~\ref{lem:rand-prop}. By Theorem~\ref{thm:main-result}, we know
that $\tilde{f}(\tilde{x}_{t}) \le \tilde{f}(\stilde) + \eps$ for all
\begin{subequations}
\begin{equation}
  \label{eqn:perturbed-time}
  \hspace{-3pt}t \geq \frac{6}{\eta}
    \left(\log\left(1+\frac{\gammaplus^{2}/4}{\rho |\tilde{b}\supind{1}|}\right)
    +\log\frac{ (\beta  + 2\rho\norm{\stilde})\norm{\stilde}^2}{\eps}\right) 
    \min\left\{\frac{1}{\rho 
    \norm{\stilde} - \gamma},
      \frac{10 \norm{\stilde}^2}{\eps}\right\},
\end{equation}
and that if  $\rho\norm{\stilde}-\gamma \le 
\frac{\eps}{10\norm{\stilde}}$, then $\tilde{f}(\tilde{x}_{t}) \le 
\tilde{f}(\stilde) + 
\eps$ for all
\begin{equation}
\label{eqn:perturbed-time-2}
t \geq \frac{6}{\eta}
\left(\log\left(1+\frac{\gammaplus^{2}}{4\rho |\tilde{b}\supind{1}|}\right)
+\log\frac{ (\beta  + 2\rho\norm{\stilde})\norm{\stilde}^2}{\eps}\right) 
\sqrt{
\frac{10 
\norm{\stilde}^2}{\eps}\frac{1}{\min\{\gap , \rho\norm{\stilde}\}}}.
\end{equation}
\end{subequations}
We now turn to bounding expressions~\eqref{eqn:perturbed-time} and 
\eqref{eqn:perturbed-time-2} appropriately:  
Section~\ref{sec:outside-log-part} deals with the occurrences of 
$\norm{\stilde}$ outside the logarithm, and  
Section~\ref{sec:inside-log-part} bounds the terms $\tilde{b}\supind{1}$ 
and $\norm{\stilde}$ appearing inside the logarithm.

\subsubsection{Part 1: upper bounding terms outside the log}
\label{sec:outside-log-part}
Recalling that $\sigma = \frac{\rho\sigbar\eps^2}{2(10\beta+20\rs)^2\norm{\s}^2}$ and 
$\eps \le (\half\beta + \rs)\norm{\s}^2$, we have $\sigma \le 
{\frac{1}{800}}\sigbar\rho\norm{\s}^2$. Thus, part (iii) of 
Lemma~\ref{lem:rand-prop} gives
\begin{equation*}
  |\norm{\s} - \norm{\stilde}| \le \sqrt{2\sigma/\rho} \le \sqrt{\sigbar} 
  \norm{\s}/20,
  ~~ \mbox{so} ~~
  \norm{\stilde} \in(1 \pm \sqrt{\sigbar}/20) \norm{\s}.
\end{equation*}
Similarly, we have
\begin{equation}\label{eq:stilde-bound-tight}
 \big| \norm{\s} - \norm{\stilde}\big|
\le \sqrt{\frac{\sigbar\eps^2}{(10\beta+20\rs)^2\norm{\s}^2}}
\le \frac{\sqrt{\sigbar}}{20}\frac{\eps }{\rho\nss}~.
\end{equation}
Now, suppose that $\frac{\eps}{10 \norm{\s}^2} \le \rho \norm{\s} - \gamma$.
Substituting this into the bound~\eqref{eq:stilde-bound-tight} yields
$|\norm{\s} - \norm{\stilde}|
\le \frac{\sqrt{\sigbar}}{2\rho}(\rs-\gamma)$, and
rearranging, we obtain
\begin{equation*}
  \rho\norm{\stilde}- \gamma \ge \left(1-0.5\sqrt{\sigbar}\right)(\rs-\gamma) 
  \ge 
  \frac{\rs-\gamma}{1+\sqrt{\sigbar}}
\end{equation*}
because $\sigbar \le 1$.
We combine the preceding bounds to obtain
\begin{subequations}
\begin{equation}
  \label{eqn:upper-bound-min-stuff}
  \min\left\{\frac{1}{\rho \norm{\stilde} - \gamma},
    \frac{10 \norm{\stilde}^2}{\eps}\right\}
  \le (1+\sqrt{\sigbar})^2\min\left\{\frac{1}{\rho \norm{{\s}} - \gamma},
  \frac{10 \norm{{\s}}^2}{\eps}\right\}
\end{equation}
and
\begin{equation}
\label{eqn:upper-bound-min-stuff-2}
\sqrt{\frac{10 
	\norm{\stilde}^2}{\eps}\cdot\frac{1}{\min\{\gap , 
	\rho\norm{\stilde}\}}}
\le
(1+\sqrt{\sigbar})\sqrt{\frac{10 
	\norm{{\s}}^2}{\eps}\cdot\frac{1}{\min\{\gap , \rs\}}} \ .
\end{equation}
\end{subequations}

The bound \eqref{eq:stilde-bound-tight} also implies
$\rho \norm{\stilde} - \gamma \le  \rs - \gamma + 
\frac{\sqrt{\sigbar}\eps}{20\nss}$.
When $\rs-\gamma \le (1-2\sqrt{\sigbar}/3)\frac{\eps}{10\nss}$, we thus have
\begin{equation*}
\rho \norm{\stilde} - \gamma
 \le \frac{\eps}{10\nss}\left( {1-\frac{2}{3}\sqrt{\sigbar}} +\frac{1}{2}\sqrt{\sigbar}\right)
 \le \frac{\eps(1+\sqrt{\sigbar}/20)^2(1-\sqrt{\sigbar}/6)}{10\norm{\stilde}^2}
 \le \frac{\eps}{10\norm{\stilde}^2},
\end{equation*}
where we have used $\norm{\stilde} \le (1+\sqrt{\sigbar}/20)\ns$ and $\sigma\le 1$. 
Therefore, $\tilde{f}(\tilde{x}_{t}) \le \tilde{f}(\stilde) + 
\eps$ whenever the conditions $\rs-\gamma \le 
(1-2\sigbar/3)\frac{\eps}{10\nss}$ and~\eqref{eqn:perturbed-time-2} hold.

\subsubsection{Part 2: upper bounding terms inside the log}
\label{sec:inside-log-part}
  Fix a confidence 
level  $\delta \in (0, 1)$. By Lemma~\ref{lem:rand-prop}(i),
$1 / |\tilde{b}^{\left(1\right)}| \leq \sqrt{2 d} / 
(\sqrt{\pi}\sigma\delta)\leq
\sqrt{d}/(\sigma\delta)$ with 
probability at least $1-\delta$, so
\begin{flalign*}
  6\log \left(1+\frac{\gammaplus^2}{4 \rho |\tilde{b}\supind{1}|}\right)
  & \le
  6 \log \left(1 + \frac{\gammaplus^2 \sqrt{d}}{4 \rho \sigma \delta}\right)
  \stackrel{(\star)}{\le} 6 \log \left(1 +\I_{\{\gamma>0\}}\frac{50\sqrt{d}}{ 
	\sigbar 
	\delta}\right) +
\\
 & 12\log\frac{ (\beta + 2\rs) \nss}{\eps}
  = \tiltgrow(d,\delta,\sigbar) + \frac{12}{20} \tiltconv(\eps),
\end{flalign*}
where inequality $(\star)$ uses that $\rs \ge \gammaplus$ and 
$\eps \le (\beta + \half \rs)\nss$.
Using $\norm{\stilde} \le (1+\sqrt{\sigbar}/20)\norm{\s}$ yields the upper 
bound
\begin{equation*}
  6\log \frac{ (\beta + 2\rho \norm{\stilde}) \norm{\stilde}^2}{\eps}
  \le 
  6\log \frac{ (\beta + 2\rs) \norm{{\s}}^2}{\eps}
  + 18\log(1+\sqrt{\sigbar}/20)
  \le \frac{8}{20}\tiltconv(\eps),
\end{equation*}
where the second inequality follows as $18\log(1+\sqrt{\sigbar}/20)<2\log2\le 
2\log\frac{ (\beta + 2\rs) \norm{{\s}}^2}{\eps}$.

Substituting the above bounds and
the upper bounds~\eqref{eqn:upper-bound-min-stuff} 
and~\eqref{eqn:upper-bound-min-stuff-2}
into expressions~\eqref{eqn:perturbed-time} 
and~\eqref{eqn:perturbed-time-2},
we see that the iteration bounds claimed in 
Theorem~\ref{thm:pert-main-result} hold. To complete
the proof we need only bound the quality of the solution $\tilde{x}_t$.

\subsubsection{Part 3: bounding solution quality}
\label{sec:solution-quality-part}
 We recall that $\sigma = 
 \frac{\rho\sigbar\eps^2}{2(10\beta+20\rs)^2\norm{\s}^2}
 \le \frac{\sigbar \eps}{800\ns}$ and 
$\norm{\stilde} \le 1+\sqrt{\sigbar}/20 \norm{\s} \le 
2\norm{\s}$, so $\sigma \le 
\frac{\sigbar\eps}{\norm{\s}+\norm{\stilde}}$. Thus, whenever 
$\tilde{f}(\tilde{x}_t) \le \tilde{f}(\stilde) + \eps$,
\begin{align*}
  f(\tilde{x}_{t})
  & \overset{\text{(a)}}{\leq} \tilde{f}(\tilde{x}_{t}) + \sigma\norm{\tilde{x}_{t}}
    \leq\tilde{f}(\stilde) + \eps +\sigma\norm{\tilde{x}_{t}}
    \overset{\text{(b)}}{\leq} \tilde{f}(\stilde) + \eps
    +\sigma\norm{\stilde} \\
  & \overset{\text{(c)}}{\leq}\tilde{f}(\s) + \eps +\sigma\norm{\stilde}
    \overset{\text{(d)}}{\leq} f(\s) + \sigma(\norm{\stilde} + \norm{\s}) + 
    \eps
  \le f(\s) + (1+\sigbar)\eps,
\end{align*}
where transitions (a) and (d) follow from part (ii) of  
Lemma~\ref{lem:rand-prop}, transition (b) follows from $\norm{\tilde{x}_t} \le 
\norm{\stilde}$ (Lemma~\ref{lemma:monotone}), and transition (c) 
follows from  $\tilde{f}(\stilde) = \min_{z\in \R^d}\tilde{f}(z)$.
\section{Convergence of a line search method}
\label{sec:linesearch}

The maximum step size allowed by
Assumption~\ref{assu:step-size} may be too conservative
(as is frequent with gradient descent).
With that in mind, in this section we briefly analyze
line search schemes of the form
\begin{equation}
  \label{eq:grad-variablestep}
  x_{t+1} = x_{t} - \eta_t \grad f(x_t)~~\mbox{where}~~
  \eta_t = \argmin_{\eta \in  C_t} f(x_t - \eta \grad f(x_t) )
\end{equation}
and $C_t$ is a (possibly time-varying) interval of allowed step sizes. For the 
problem~\eqref{eq:problem}, $\eta_t$ is computable for any interval 
$C_t$, as the critical points of the function $h(\eta) = f(x_t - \eta \grad 
f(x_t) )$ are roots of a quatric polynomial with coefficients determined by
$\norm{x}, \norm{g}, g^T A g$, and $x^T g$, so $\eta_t$ must be 
a root or an edge of the interval $C_t$.

The unconstrained choice $C_t = \R$ yields the \textit{steepest descent}
method~\cite{NocedalWr06}.
For steepest descent it is possible that $\eta_t < 0$ and that 
convergence to a sub-optimal local minimum of $f$ occurs.
Consequently,
we propose choosing the updates~\eqref{eq:grad-variablestep}
using the interval
\begin{equation}\label{eq:constrained-steepest}
  C_t = \left[0,\hinge{
      \frac{\grad f(x_t)^{\T} A \grad f(x_t)}{\|\grad f(x_t)\|^2}
      + \rho \|x_t\|}^{-1}\right].
\end{equation}
The scheme~\eqref{eq:constrained-steepest} converges to the 
global minimum of $f$ (see Appendix~\ref{app:linesearch} for proof):

\begin{restatable}{proposition}{propLinesearch}\label{prop:linesearch}
  Let $x_t$ be the iterates of gradient descent with step sizes selected by
  the constrained minimization~\eqref{eq:constrained-steepest}. Let
  Assumption~\ref{assu:init} hold and assume $b^{(1)}\neq 0$.  Then $\s$ is the
  unique global minimizer of $f$ and $\lim_{t\to\infty}x_t = \s$.
\end{restatable}

\begin{figure}
  \begin{center}
    \includegraphics[width=0.9\columnwidth]{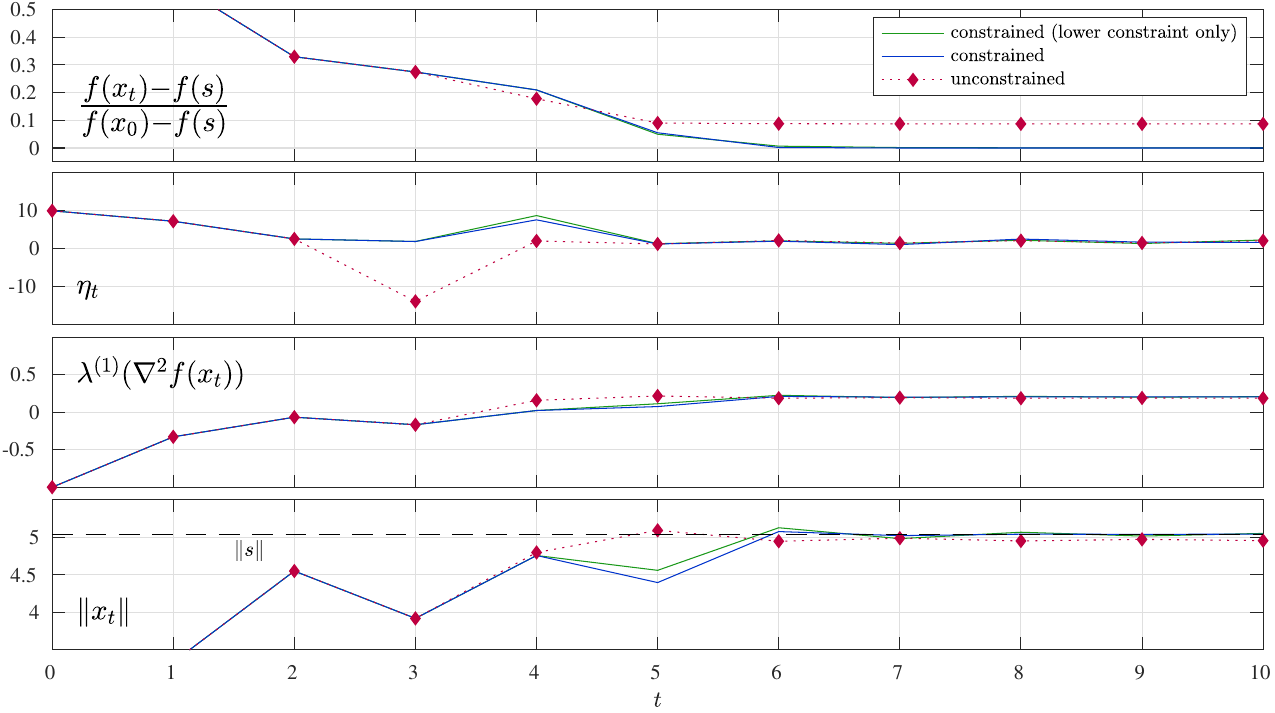}
    \caption{\label{fig:steepest-fail}
      Steepest descent variants applied on $A = \diag ([-1; -0.8; -0.5])$, $b
      = [0.04; 0.15; 0.3]$ 
      and $\rho = 0.2$. The red, green, and blue curves correspond to $C_t = 
      \R$, $C_t = [0,\infty)$ and $C_t$ given 
      by~\eqref{eq:constrained-steepest}, respectively. }
  \end{center}
\end{figure}

In Fig.~\ref{fig:steepest-fail}, we display the quantities
$f(x_t), \eta_t, \lambda^{(1)}(\hess f(x_t))$, and $\|x_t\|$ for
the above
line-search variants 
on a $d= 3$-dimensional problem instance. The step sizes differ at
iteration $t=3$, where the unconstrained gradient step makes almost $50\%$ 
more progress than steps
restricted to be positive. However, it then converges to a 
sub-optimal local minimum (note $\lambda^{(1)}(\hess f(x_t)) > 
0$) approximately $9\%$ worse than the global minimum achieved by the guarded
sequence~\eqref{eq:constrained-steepest}. The step sizes these
methods choose are significantly larger than the $\eta$
Assumption~\ref{assu:step-size} allows, which is approximately $0.12$.
Fig.~\ref{fig:steepest-fail} reveals a difference between fixed step
size gradient descent and the line-search schemes---the
norm $\norm{x_t}$ of the line-search-based iterates is non-monotonic and
overshoots $\norm{\s}$. Our convergence rate proofs hinge on
Lemma~\ref{lemma:monotone}, that $\norm{x_t}$ is increasing, so
extension of our rates to line-search schemes is not
straightforward. 

\begin{figure}
  \begin{center}
    \heavyplot{\includegraphics[width=0.9\columnwidth]{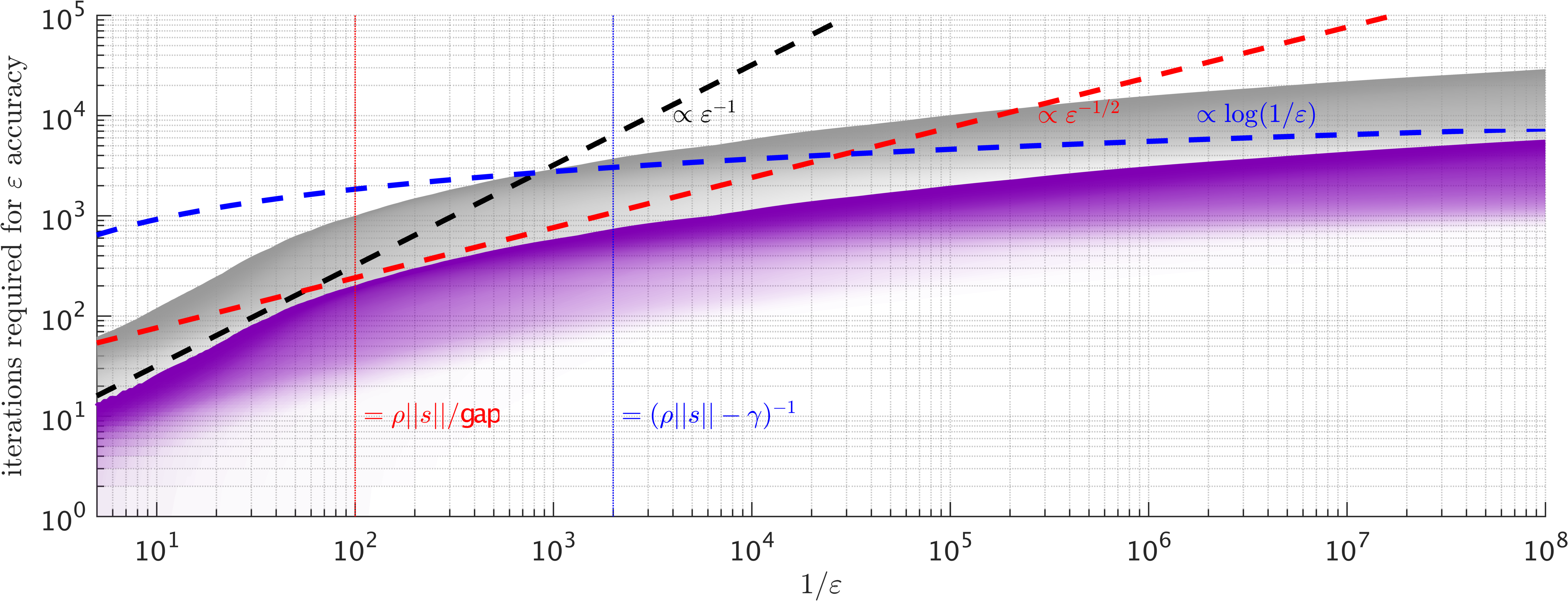}}
    \caption{\label{fig:ensembleLS}
      Reproduction of the ensemble experiment reported in Section
      \ref{sec:experiments} (Figure \ref{fig:ensemble}), with the scheme
      \eqref{eq:constrained-steepest} used instead of fixed step size gradient
      descent. The cdf for fixed step size $\eta = 0.25$ is shown in gray for
      comparison.}
  \end{center}
\end{figure}

We believe that the rate guarantees of Theorem \ref{thm:main-result} apply
also to the step size choice~\eqref{eq:constrained-steepest}. To lend credence
to this hypothesis, we repeat the ensemble experiment detailed in
Section~\ref{sec:experiments} (Figure~\ref{fig:ensemble}), where we use the
step size~\eqref{eq:constrained-steepest} instead of the fixed step size.
Figure~\ref{fig:ensembleLS} shows that the rates we prove in 
Section~\ref{sec:mainresult} seem to accurately describe the 
behavior of guarded steepest descent as well, with constant factors.

We remark that we introduce the upper
constraint~\eqref{eq:constrained-steepest} only because we require it in the
proof of Proposition~\ref{prop:linesearch}. Empirically, a scheme with the
simpler constraint $C_t = [0, \infty)$ appears to converge to the global
minimum as well, though we remain unable to prove this. While such step size
can differ from the choice~\eqref{eq:constrained-steepest} (see time $t = 4$
in Fig.~\ref{fig:steepest-fail}), the variants seem equally practicable.
Indeed, we performed the ensemble experiment
(Figs.~\ref{fig:ensemble} and~\ref{fig:ensembleLS}) with $C_t = [0, \infty)$
and the results are
indistinguishable.
\section{Application: A Hessian-free majorization method}
\label{sec:trustregion}

In this section we use our main results to analyze a simple optimization 
scheme that approximates the cubic-regularized Newton steps with 
gradient descent.  We expect more elaborate schemes to be more efficient 
in practice, as the current procedure is both simplified and uses gradient 
descent rather than Krylov subspace methods (see the introduction). We 
believe that our analysis extends to such practical schemes beyond the 
scope of this paper.

 We consider functions $g$ satisfying the
following %

\begin{assumption}
  \label{assu:g}

The function $g$ satisfies $\inf g \ge \glb > -\infty$, is 
$\beta$-smooth and 
has $2\rho$-Lipschitz Hessian, \ie$\|\hess g(y) - \hess g(y')\| \le 
2\rho \|y-y'\|$ for every $y,y'\in\R^d$.
\end{assumption}

The first two parts of Assumption~\ref{assu:g} (boundedness and 
smoothness) 
are standard.  The third implies that $g$ is upper bounded by its
cubic-regularized quadratic approximation~\cite[Lemma~1]{NesterovPo06}:
for $y, \Delta \in \R^d$ one has
\begin{equation}
  \label{eq:majorization}
  g(y + \Delta) \le g(y) + \grad g(y)^{\T}\Delta + \half \Delta^{\T} \hess g(y) \Delta
  + \frac{\rho}{3} \norm{\Delta}^3.
\end{equation}
For simplicity we assume that the constants $\beta$ and $\rho$ are 
known. From 
a theoretical perspective this is a benign assumption, as we may estimate 
these constants without significantly affecting the complexity  
bounds~\cite{Nesterov07}. In practice, however, careful adaptive estimation 
of $\rho$ is 
crucial for good performance; this is a primary strength of the ARC 
method~\cite{CartisGoTo11}.

Following ~\cite{NesterovPo06, GeHuJiYu15}, our goal is to find an
$\epsilon$-second-order stationary point $y_\epsilon$:
\begin{equation}\label{eq:second-order-crit}
  \norm{\grad g(y_\epsilon)}
  \le \epsilon ~~\textrm{and} ~~
  \hess g(y_\epsilon) \succeq -\sqrt{\rho\epsilon}I.
\end{equation} 
Intuitively, $\epsilon$-second-order stationary points provide a finer 
approximation to local minima than $\epsilon$-stationary points (with only 
$\norm{\grad g(y)} \le \epsilon$).
Throughout this section,
we use $\epsilon$ to denote approximate stationarity in $g$, and continue to
use $\eps$ to denote approximate optimality for subproblems of the form
\eqref{eq:problem}.%

We outline a majorization-minimization
strategy~\cite{ConnGoTo00,NocedalWr06} for optimization of $g$
in Algorithm~\ref{alg:SP}.
At each iteration, the
method approximately minimizes a local model of $g$,
halting once progress decreasing $g$ falls below a certain threshold.  
In Algorithm~\ref{alg:SSP}, we describe a simple Hessian-free subproblem 
solver that uses gradient descent 
with a small perturbation to the linear term and \emph{fixed step size} (as 
in 
Theorem~\ref{thm:pert-main-result}); we write the method in terms of an
input matrix $A = \nabla^2 g(y)$, noting that it requires only
matrix-vector products $Av$ implementable by a first-order oracle for $g$.

The method \callSSP{} takes as input a problem instance $(A,b,\rho)$,
confidence level $\delta$, relative accuracy $\eps'$, and a threshold for the 
magnitude of the global minimizer $\s$, which we denote by
$\slb$. As an immediate consequence of Theorem~\ref{thm:pert-main-result}, 
as
long as $\|\s\|\ge\slb$ the method is guaranteed to terminate before reaching
line~\ref{line:SSPfail}, and if the gradient is sufficiently large,
termination occurs before entering the loop. We formalize this in the
following lemma, whose proof we provide in Appendix~\ref{sec:proof-SSP}.

\begin{algorithm}[t]
  \caption{A second-order majorization method}\label{alg:SP}
  \begin{algorithmic}[1]
    \Function{\SP}{$y_0$, $g$, $\beta$, $\rho$, $\epsilon$, 
    $\delta$}\label{func:SP}
    \State Set $\Kprog = 1/324$ and $\eta  = 1/(10\beta)$%
    \For {$k = 1, 2, \ldots$} \Comment{guaranteed to terminate in at most  
    $O(\epsilon^{-3/2})$ iterations}
    \State $\Delta_k \gets \callSSP{\hess g(y_{k-1}),\  \grad g(y_{k-1}),\  \rho,\  
      \eta,\   \sqrt{\frac{\epsilon}{9\rho}},\ \half,\ \frac{\delta}{2k^2}}$
    \If{ $g(y_{k-1} + \Delta_k) \le g(y_k) - \Kprog \epsilon^{3/2} 
    \rho^{-1/2}$}\label{line:progress}
    \State $y_k \gets y_{k-1} + \Delta_k$
    \Else
    \State $\Delta_k \gets \callSFSP{\hess g(y_{k-1}),\  \grad g(y_{k-1}),\  
    \rho,\  \eta,\  \frac{\epsilon}{2}}$
    \State \Return{$y_{k-1} + \Delta_k$}
    \EndIf
    \EndFor
    \EndFunction
  \end{algorithmic}
\end{algorithm}

\begin{algorithm}[t]
	\caption{A Hessian-free subproblem solver}\label{alg:SSP}
	\begin{algorithmic}[1]
		\Function{\SSP}{$A$, $b$, $\rho$, $\eta$, $\slb$, $\eps'$, 
		$\delta$}\label{func:SSP}
		\State Set $f(x) = (1/2)x^{\T}Ax + b^{\T}x +(\rho/3)\|x\|^3$
                and $x_0= \callCP{A, b, \rho}$ \label{line:x0}
		\If{$f(x_0) \le -(1-\eps')\rho \slb^3/6$}\label{line:x0if}
		\Return{$x_0$} \EndIf
		\State Set \label{line:T}
		\[
		T = 
		\frac{480}{\eta\rho\slb\eps'}\bigg[
                6\log\bigg(1+\frac{\sqrt{d}}{\delta}\bigg)
                + 
		44\log\left(\frac{6}{\eta\rho\slb\eps'}\right)\bigg]
		\]
		\State Set $\sigma = \frac{\rho^3 \slb^4 (\eps')^2}{2(120\beta+240\rho\slb)^2}$, 
		draw 
		$q$ uniformly from 
		the unit sphere, set $\tilde{b} = b + \sigma q$
		\State Set $\tilde{f}(x) = (1/2)x^{\T}Ax + \tilde{b}^{\T}x +(\rho/3)\|x\|^3$
                and $\tilde{x}_0= \callCP{A, \tilde{b}, \rho}$
		
		\For{$t = 1, 2, \ldots, T$}
		\State $\tilde{x}_t \gets \tilde{x}_{t-1} - \eta \grad \tilde{f}(\tilde{x}_{t-1})$
		\If{$f(\tilde{x}_t) \le -(1-\eps')\rho \slb^3/6$} 
		\Return{$\tilde{x}_t$}
		\EndIf
		\EndFor
		\State \Return{$\tilde{x}_t$}\label{line:SSPfail}
		\EndFunction
	\end{algorithmic}
	\begin{algorithmic}[1]
	    \Function{\SFSP}{$A$, $b$, $\rho$, $\eta$, 
	    $\epsgrad$}\label{func:SFSP}
		\State Set $f(x) = (1/2)x^{\T}Ax + b^{\T}x +(\rho/3)\|x\|^3$
		and $\Delta = \callCP{A, b, \rho}$
		\While{$\norm{\nabla f(\Delta)} > \epsgrad$} \label{line:sfsp-while}
		$\Delta \gets \Delta - \eta \nabla f(\Delta)$
		\EndWhile
		\State \Return{$\Delta$}
		\EndFunction
	\end{algorithmic}
	\begin{algorithmic}[1]
		\Function{\CP}{$A$, $b$, $\rho$}\label{func:CP}
		\vspace{-0.2cm}
		\State \Return $-R_c b /\norm{b}$ where $R_{c} =   
		\frac{-b^{\T}Ab}{2\rho\| b\| ^{2}}+\sqrt{\left(\frac{b^{\T}Ab}{2\rho\| b\| 
		^{2}}\right)^{2}+\frac{\| b\| }{\rho}}$
		\EndFunction
	\end{algorithmic}
\end{algorithm}

\begin{restatable}{lemma}{lemSSP}\label{lem:SSP}
  Let $A \in \R^{d\times d}$ satisfy $\opnorm{A} \le \beta$, $b\in\R^d$, 
  $\rho > 0, 
  \slb>0$, $ \eps' \in (0, 1), \delta \in (0,1)$ and $\eta \le 1/(8\beta + 
  4\rho\slb)$. With 
  probability at least $1-\delta$, if
  \begin{equation*}
    \|\s\| \ge \slb \textrm{ or } \|b\| \ge \max\{\sqrt{\beta\rho}\slb^{3/2}, 
    \rho\slb^2\}
  \end{equation*}
  then $x=\text{\callSSP{$A$, $b$, $\rho$, $\eta$, $\slb$, $\eps'$, 
      $\delta$}}$ satisfies 
  $f(x) \le (1-\eps')\rho\slb^3/6$.
\end{restatable}

Let $\siter_k$ be the global minimizer (in $\Delta$) of the
model~\eqref{eq:majorization} at $y=y_k$, the $k$th iterate of
Algorithm~\ref{alg:SP}. Lemma~\ref{lem:SSP} guarantees that with high
probability, if \callSSP{} fails to meet the progress
condition in line~\ref{line:progress}  at iteration $k$, then $\|\siter_k\| \le
\sqrt{\epsilon/(9\rho)}$, and therefore $\lambda^{(1)}(\hess g(y_k)) \ge
-\rho \|\siter_k\| \ge -\sqrt{\rho\epsilon}$. 
It is possible, nonetheless, that $\norm{\grad g(y_k)} > \epsilon$;
to address this, we correctively apply gradient descent on the final 
subproblem (\callSFSP{}).

Building off of an argument
of \citet[Lemma 5]{NesterovPo06}, we obtain the following guarantee
for Algorithm~\ref{alg:SP}, whose proof we provide in
Appendix~\ref{sec:proof-SP}.

\begin{restatable}{proposition}{propSP}\label{prop:SP}
  Let $g$ satisfy Assumption \ref{assu:g},
  $y_0\in\R^d$ be arbitrary, and let
  $\delta \in (0, 1]$ and $\epsilon \le 
  \min\{\beta^2/\rho,\rho^{1/3}(g(y_0)-\glb)^{2/3}\}$. With
  probability at least $1-\delta$, Algorithm \ref{alg:SP} finds an
  $\epsilon$-second-order stationary point~\eqref{eq:second-order-crit} in at
  most 
  \begin{equation}
    O(1) \cdot \frac{\beta(g(y_0)-\glb)}{\epsilon^{2}}
    \log\left(\frac{d}{\delta}\cdot\frac{\beta(g(y_0)-\glb)}{\epsilon^{2}}\right)
     \label{eq:trustregion-complexity}
  \end{equation}
  Hessian-vector product evaluations, and at most
  \begin{equation*}
  O(1) \cdot \frac{\sqrt{\rho}(g(y_0)-\glb)}{\epsilon^{3/2}}
  \end{equation*}
  calls to \callSSP{} and gradient evaluations.
\end{restatable}

In Proposition~\ref{prop:SP}, the assumption $\epsilon \le \beta^2/\rho$ is no
loss of generality, as otherwise the Hessian
guarantee~\eqref{eq:second-order-crit} is trivial, and we may obtain the
gradient guarantee by simply running gradient descent on $g$ for
$2\beta(g(y_0)-\glb)\epsilon^{-2}$ iterations. Similarly, if
$\epsilon > \rho^{1/3}(g(y_0)-\glb)^{2/3}$ then we require at most 
$1+1/\Kprog=325$ calls to \callSSP{}, and the proof
of Proposition~\ref{prop:SP} reveals that the overall first-order complexity  
scales
as $\epsilon^{-1/2}$ instead of $\epsilon^{-2}$.

There are other Hessian-free methods that provide the 
guarantee~\eqref{eq:second-order-crit}, and recent schemes 
using acceleration techniques~\cite{AgarwalAlBuHaMa17, 
CarmonDuHiSi18} %
provide it in roughly $\epsilon^{-7/4} \log \frac{d}{\delta}$ first-order 
operations,
which is better than 
Algorithm~\ref{alg:SP}. Nevertheless, this section illustrates how gradient 
descent on the structured problem~\eqref{eq:problem} can be 
straightforwardly leveraged to optimize general smooth non-convex functions.

\section{Discussion}\label{sec:discussion}

Our results have a number of connections to rates of convergence in classical 
(smooth) convex
optimization and the power method for symmetric eigenvector
computation; here, we explore these in more detail.

\subsection{Comparison with convex optimization}

For $L$-smooth $\alpha$-strongly convex functions,  
gradient descent finds an $\eps$-suboptimal point within 
\begin{equation*}
  O(1) \cdot \min\left\{ \frac{L}{\alpha}\log \frac{L D^2}{\eps},
    \frac{L D^2}{\eps} \right\}
\end{equation*}
iterations~\cite{Nesterov04}, where $D$ is any constant
$D \ge \norm{x_0 - x^\star}$ and $x^\star$ a global minimizer. 
For our (possibly non-convex) problem~\eqref{eq:problem},
Corollary~\ref{cor:main-result} guarantees that
gradient descent finds an
$\eps$-suboptimal point (with probability at least $1 - \delta$) within
\begin{equation*}
  O(1) \cdot
  \min \left\{
    \frac{\Ls}{\rs - \gamma}, 
    \frac{\Ls \nss}{\eps} \right\}\left[\log\frac{\Ls \nss}{ \eps}
    + \log
    \left(1 + \I_{\{\gamma>0\}}\frac{ d}{\delta}\right)
  \right]
\end{equation*}
iterations, where $\Ls = \beta + 2\rs$. The parallels are immediate: by
Lemma~\ref{lemma:monotone}, $\Ls$ and $\norm{\s}$ are precise analogues
of $L$ and $D$ in the convex setting. Moreover, the quantity $\rs - \gamma$
plays the role of the strong convexity parameter $\alpha$, but it is
well-defined even when $f$ is not convex.
When $\lambda\supind{1}(A) = -\gamma \ge 0$, $f$ is $-\gamma$-strongly convex,
and because $\rs-\gamma > -\gamma$,
our analysis for the cubic problem~\eqref{eq:problem}
guarantees better conditioning than the generic convex result.
The difference between
$\rs-\gamma$ and $-\gamma$ becomes significant when $b$ is sufficiently large,
as we observe from the bounds~\eqref{eq:s-lower-R-def} 
and~\eqref{eq:Rc-def}. Even in the non-convex case that $\gamma > 0$,
gradient descent still
exhibits linear convergence whenever high accuracy is
desired, that is, when $\eps/\nss \le \rs-\gamma$.

When $\gamma > 0$, our guarantee becomes probabilistic and contains a
$\log(d/\delta)$ term. Such a term does not appear in results on convex
optimization~\cite{Nesterov04}, and it is fundamentally related to the
presence of saddle-points in the objective~\cite{SimchowitzAlRe18}.

\subsection{Comparison with the power method}

The power method for 
finding the
smallest eigenvector of $A$ is the recursion $x_{t+1} = (I-
(1/\beta) A)x_t/\|(I-(1/\beta) A)x_t\|$ where $x_0$ is uniform on the unit
sphere  \cite{KuczynskiWo92, MuscoMu15}. This method guarantees that 
with probability at least $1-\delta$,
$x_t^{T}Ax_t \le -\gamma + \eps$ for all
\begin{equation*}
t \ge O(1)\cdot\min\left\{ \frac{\beta}{\eps}\log\left(\frac{d}{\delta}\right) ,
\frac{\beta}{\gap}\log\left(\frac{\beta}{\eps}\cdot\frac{d}{\delta}\right)\right\}.
\end{equation*}

When $b = 0$ and $\lambda\supind{1}(A) = -\gamma < 0$, any
global minimizer of problem~\eqref{eq:problem} is
an eigenvector of $A$ with eigenvalue 
$-\gamma$ and $\rs = \gamma$,
so it is natural to compare gradient descent and the power method.
For simplicity, let us assume 
that $\rho=\gamma$ so that $\norm{\s}=1$, and both methods converge 
to unit eigenvectors. Under these assumptions $f(x) = \half x^{\T}Ax + 
\frac{\gamma}{3}\norm{x}^3$ and $f(\s) = -\gamma/6$, so $f(x) \le f(\s) 
+ \eps'$ implies
\begin{equation*}
  \frac{x^{\T}Ax}{\norm{x}^2} \le -\frac{\gamma}{3}
  \left[\frac{1}{\norm{x}^2}+2\norm{x}\right] + \frac{2\eps'}{\norm{x}^2}
  \le -\gamma + \frac{2\eps'}{\norm{x}^2}.
\end{equation*}
Consider gradient descent applied to $f$ with a random perturbation as described in 
Theorem~\ref{thm:pert-main-result}, with $\sigbar=1$. Inspecting
the proofs of our theorems (Sec.~\ref{sec:proofs}), we see that
Lemmas~\ref{lem:growth} and \ref{lem:rand-prop} imply that
with probability at least $1 - \delta$ we have
$\norm{\tilde{x}_t}\ge1/2$  for every $t\ge O(1) 
\log (\frac{\Ls\nss}{\eps}\cdot\frac{d}{\delta})$. As
in Corollary~\ref{cor:main-result}, setting
$\eta=\frac{1}{4(\beta+\rho R)}=\frac{1}{8\beta}$  
guarantees
that with probability at least $1-\delta$, 
$\tilde{x}_t^{T}A\tilde{x}_t/\norm{\tilde{x}_t}^2 \le -\gamma + 
\eps$ for all
\begin{equation*}
t \ge O(1)\cdot\min\left\{ 
\frac{\beta}{\eps}\log\left(\frac{\beta}{\eps}\cdot\frac{d}{\delta}\right) ,
\frac{\beta}{\sqrt{\eps\min\{\gap,\gamma\}}}\log\left(\frac{\beta}{\eps}\cdot\frac{d}{\delta}\right)\right\}.
\end{equation*}

Comparing the rates of convergence, we see that both exhibit the
$\log (d/\delta)$ hallmark of non-convexity and gap-free and
gap-dependent convergence regimes. Of course, the power method also finds
eigenvectors when $\gamma < 0$, while the unique solution to
problem~\eqref{eq:problem} when $b = 0$ and $\gamma < 0$ is simply $\s = 
0$.
In the gap-dependent regime, however, the power method enjoys linear
convergence when $\eps < \gap$, while our bounds have a $1 / \sqrt{\eps}$ 
factor. Although this may be due to looseness in our analysis, we suspect 
it
is real and related to the fact that gradient descent needs to
``grow'' the iterates to have norm $\norm{x_t} \approx \gammaplus/\rho$, while
the power method iterates always have unit norm. If one is only interested in
finding eigenvectors of $A$, there is probably no reason to prefer the
cubic-regularized objective to the power method.

\section*{Acknowledgment}
We thank Ziyi Chen for pointing out an error in the previous version of part (iii) of Lemma \ref{lem:rand-prop}.
YC and JCD were partially supported by
the SAIL-Toyota Center for AI Research. YC was partially supported by 
the Stanford Graduate Fellowship and the Numerical Technologies  
Fellowship. JCD
was partially supported by
the National Science Foundation award NSF-CAREER-1553086.

\appendix

\toggletrue{restatements}

\part*{Appendix}
\restate{
For convenience of the
reader, throughout the appendix we restate lemmas that appeared in the 
main 
text prior to giving their proofs.
}

\section{Proof of Lemma~\ref{lem:monotone-weak}}
\label{app:prel}
Before proving Lemma~\ref{lem:monotone-weak}, we state and prove
two technical lemmas (see Sec.~\ref{sec:finally-proof-monotone-weak}
for the proof conditional on these lemmas).
For the first lemma,
let $\kappa \in \R^d$ satisfy $\kappa\supind{1} \le \kappa\supind{2}
\le \ldots \le \kappa\supind{d}$, let
$\nu_t$ be a nonnegative and nondecreasing sequence, $0 \le \nu_1 \le
\nu_2 \le \ldots$, and consider the process
\begin{equation}\label{eq:zt-def}
  z_{t}\supind{i}=(1-\kappa\supind{i}-\nu_{t-1})z_{t-1}\supind{i}+1.
\end{equation}
Additionally, assume
$1-\kappa\supind{i}-\nu_{t-1}\geq0$ for all $i$ and $t$.
\begin{lemma}
  \label{lem:ordering}
  Let $z_{0}\supind{i}=c_{0}\geq0$ for every $i \in [d]$. Then for
  every $t \in \N$ and $j \in [d]$, the following holds.
  \begin{enumerate}[(i)]
  \item \label{item:future-signs}
    If $z_{t}^{\left(j\right)}\leq z_{t-1}^{\left(j\right)}$ then also
    $z_{t'}^{\left(j\right)}\leq z_{t'-1}^{\left(j\right)}$ for every
    $t'>t$.
  \item \label{item:sign-ratio} If
    $z_{t}^{\left(j\right)}\geq z_{t-1}^{\left(j\right)}$, then
    $z_{t}^{\left(j\right)}/z_{t+1}^{\left(j\right)}\geq
    z_{t}\supind{i}/z_{t+1}\supind{i}$ for every $i\le j$.
  \item \label{item:bigger-signs}
    If $z_{t+1}\supind{i}\leq z_{t}\supind{i}$,
    then $z_{t+1}\supind{j}\leq z_{t}\supind{j}$
    for every $j\geq i$.
  \end{enumerate}
\end{lemma}
\begin{proof}
  For shorthand, we define
  $\delta_{t}\supind{i} \defeq \kappa\supind{i}+\nu_{t}$. 
  
  We first establish part~\eqref{item:future-signs} of the lemma. 
  By~\eqref{eq:zt-def}, we have
  \begin{alignat*}{1}
    z_{t+1}\supind{j} - z_{t}\supind{j}
    =(1-\delta_{t-1}^{\left(j\right)})(z_{t}^{\left(j\right)}-z_{t-1}^{\left(j\right)})- 
    (\delta_{t}^{\left(j\right)}-\delta_{t-1}^{\left(j\right)})z_{t}^{\left(j\right)}
  \end{alignat*}
  By our assumptions that $z_{0}\supind{j}\geq0$ and that
  $1-\delta_{t}\supind{j}\geq0$ for every $t$ we immediately have that
  $z_{t}\supind{j}\geq0$, and therefore also
  $(\delta_{t}\supind{j}-\delta_{t-1}\supind{j})z_{t}\supind{j} = 
  (\nu_{t}-\nu_{t-1})z_{t}\supind{j}\geq0$.
  We therefore conclude that
  \begin{equation*}
  z_{t+1}\supind{j}-z_{t}\supind{j}\leq (1-\delta_{t-1}\supind{j})  
  (z_{t}\supind{j}-z_{t-1}\supind{j} ) \le 0,
  \end{equation*}
  and induction gives part~\eqref{item:future-signs}.
  
  To establish part~\eqref{item:sign-ratio} of the lemma, first note
  that by the contrapositive of part~\eqref{item:future-signs},
  $z_{t}\supind{j}\geq z_{t-1}\supind{j}$
  for some $t$ implies $z_{t'}\supind{j}\geq z_{t'-1}\supind{j}$
  for any $t'\leq t$. We prove by induction that
  \begin{equation}
    z_{t'}\supind{i}-z_{t'}\supind{j}
    \leq (\kappa\supind{j}-\kappa\supind{i}) z_{t'}\supind{i}z_{t'}\supind{j}
    \label{eqn:kappa-z-inequality}
  \end{equation}
  for any $i\leq j$ and $t'\leq t$. The basis of the induction is
  immediate from the assumption $z_{0}\supind{i}=z_{0}\supind{j}\ge0$.
  Assuming the property holds through time $t' - 1$ for $t' \le t$, we
  obtain
  \begin{alignat*}{1}
    \frac{z_{t'}\supind{i}-z_{t'}\supind{j}}{z_{t'}\supind{i}z_{t'}\supind{j}}
    & =\frac{(1-\delta_{t'-1}\supind{i})(z_{t'-1}\supind{i}-z_{t'-1}\supind{j})+
      (\delta_{t'-1}\supind{j}-\delta_{t'-1}\supind{i})z_{t'-1}\supind{j}}{
      z_{t'}\supind{i}z_{t'}\supind{j}} \\
    & \le \frac{(1 - \delta_{t'-1}\supind{i}) (\kappa\supind{j} - \kappa\supind{i})
      z_{t'-1}\supind{i} z_{t'-1}\supind{j}}{z_{t'}\supind{i} z_{t'}\supind{j}}
    = (\kappa\supind{j} - \kappa\supind{i}) \frac{z_{t'-1}\supind{j}}{
      z_{t'}\supind{j}} \le \kappa\supind{j} - \kappa\supind{i}
  \end{alignat*}
  where the first inequality uses inequality~\eqref{eqn:kappa-z-inequality} 
  (assumed by induction)
  and the second uses $z_{t'-1}\supind{j} \le
  z_{t'}\supind{j}$ for any $t'\le t$, as argued above.
  With the bound
  $z_{t}\supind{i}-z_{t}\supind{j}\leq(\kappa\supind{j}-\kappa\supind{i})
  z_{t}\supind{i}z_{t}\supind{j}$
  in place, we may finish the proof of part~\eqref{item:sign-ratio} by noting that
  \begin{alignat*}{1}
    \frac{z_{t}\supind{j}}{z_{t+1}\supind{j}}-\frac{z_{t}\supind{i}}{z_{t+1}\supind{i}}
    & =\frac{z_{t+1}\supind{i}z_{t}\supind{j}-z_{t+1}\supind{j}z_{t}\supind{i}}{z_{t+1}\supind{j}z_{t+1}\supind{i}}
    =\frac{(\kappa\supind{j}-\kappa\supind{i})z_{t}\supind{i} 
    z_{t}\supind{j}-(z_{t}\supind{i}-z_{t}\supind{j})}{z_{t+1}\supind{j} 
    z_{t+1}\supind{i}}\geq0.
  \end{alignat*}

  Lastly, we prove part~\eqref{item:bigger-signs}. If
  $z_{t}\supind{j}\leq z_{t-1}\supind{j}$ then we have
  $z_{t+1}\supind{j}\leq z_{t}\supind{j}$ by
  part~\eqref{item:future-signs}. Otherwise we have
  $z_{t}\supind{j}\ge z_{t-1}\supind{j}$, and so
  $z_{t}\supind{j}/z_{t+1}\supind{j}\geq z_{t}\supind{i}/z_{t+1}\supind{i}$ by
  part~\eqref{item:sign-ratio}. As $z_{t+1}\supind{i}\leq z_{t}\supind{i}$,
  this implies
  $z_{t}\supind{j} / z_{t+1}\supind{j} \geq
  z_{t}\supind{i} / z_{t+1}\supind{i} \geq 1$
  and therefore $z_{t+1}\supind{j}\leq z_{t}\supind{j}$
  as required.
\end{proof}

Our second technical lemma provides a lower bound on certain
inner products in the gradient descent iterations.
In the lemma, we recall the definition~\eqref{eq:R-def} of
$R$.
\begin{lemma}
  \label{lemma:x-transpose-A-nice}
  Assume that $\norm{x_\tau}$ is non-decreasing in $\tau$ for
  $\tau \le t$, that $\norm{x_t} \le R$, and that
  $x_t^{\T}\nabla f(x_t) \le 0$. Then $x_t^{\T}A\nabla f(x_t) \ge \beta 
  x_t^{\T}\nabla 
  f(x_t)$.
\end{lemma}
\begin{proof}
  If we define
  $z_t \supind{i}=x_t \supind{i}/(-\eta b\supind{i})$, then
  evidently
  \begin{equation*}
    z_{t + 1}\supind{i}=(1-\underset{\triangleq\kappa\supind{i}}{\underbrace{\eta\lambda\supind{i}\left(A\right)}}-\underset{\triangleq\nu_{t}}{\underbrace{\eta\rho\norm{ x_{t}} }})z_{t}\supind{i}+1.
  \end{equation*}
  We verify that $z_t\supind{i}$ satisfies
  the conditions of Lemma~\ref{lem:ordering}:
  \begin{enumerate}[i.]
  \item By definition $\kappa\supind{i}$ are increasing in $i$, and
    $\nu_{0}\leq\nu_{1}\leq\cdots\leq\nu_{t}$ by our assumption
    that $\norm{ x_\tau} $ is non-decreasing for $\tau \le t$.
  \item As
    $\eta\leq1/\left(\beta+\rho R\right)$
    for $\tau \leq t$, we have that $\kappa\supind{i}+\nu_\tau \leq1$ for
    $\tau \le t$ and $i \in [d]$.
  \item As $x_{0}=-r b/\norm{ b}$, $z_{0}\supind{i}=r/(\eta\norm{ b}) \geq0$
    for every $i$.
  \end{enumerate}
  We may therefore apply Lemma~\ref{lem:ordering},
  part~\eqref{item:bigger-signs} to conclude that
  $z_{t}\supind{i}-z_{t + 1}\supind{i}\geq0$ implies
  $z_{t}\supind{j}-z_{t + 1}\supind{j}\geq0$ for every $j\geq i$. Since
  $z_{t}\supind{i}\geq0$ for every $i$,
  \begin{equation*}
  \mathrm{sign}\left(x_{t}\supind{i}\left(x_{t}\supind{i}-x_{t + 1}\supind{i}\right)\right)=\mathrm{sign}\left(z_{t}\supind{i}\left(z_{t}\supind{i}-z_{t + 1}\supind{i}\right)\right)=\mathrm{sign}\left(z_{t}\supind{i}-z_{t + 1}\supind{i}\right),
  \end{equation*}
  and there must thus exist some $i^{*}\in[d]$ such that
  $x_{t}\supind{i}(x_{t}\supind{i}-x_{t + 1}\supind{i})\leq0$ for every
  $i\leq i^{*}$ and
  $x_{t}\supind{i}(x_{t}\supind{i}-x_{t + 1}\supind{i})\geq0$ for every
  $i>i^{*}$. We thus have (by expanding
  in the eigenbasis of $A$) that
  \begin{alignat*}{1}
    & x_{t}^{\T}A\grad f\left(x_{t}\right)  
    =\frac{1}{\eta}\sum_{i=1}^{i^{*}}\lambda\supind{i}\left(A\right)x_{t}\supind{i}\left(x_{t}\supind{i}-x_{t
     + 
    1}\supind{i}\right)+\frac{1}{\eta}\sum_{i=i^{*}+1}^{d}\lambda\supind{i}\left(A\right)x_{t}\supind{i}\left(x_{t}\supind{i}-x_{t
     + 1}\supind{i}\right)\\
    & \quad 
    \geq\lambda^{\left(i^{*}\right)}\left(A\right)\frac{1}{\eta}\sum_{i=1}^{i^{*}}x_{t}\supind{i}\left(x_{t}\supind{i}-x_{t
     + 
    1}\supind{i}\right)+\lambda^{\left(i^{*}+1\right)}\left(A\right)\frac{1}{\eta}\sum_{i=i^{*}+1}^{d}x_{t}\supind{i}\left(x_{t}\supind{i}-x_{t
     + 1}\supind{i}\right)\\
    & \quad 
    \geq\lambda^{\left(i^{*}\right)}\left(A\right)\frac{1}{\eta}\sum_{i=1}^{d}x_{t}\supind{i}\left(x_{t}\supind{i}-x_{t
     + 1}\supind{i}\right)=\lambda^{\left(i^{*}\right)}\left(A\right)x_{t}^{\T}\grad 
    f\left(x_{t}\right)\geq \beta x_{t}^{\T}\grad f\left(x_{t}\right)
  \end{alignat*}
  where the first two inequalities use the fact the $\lambda\supind{i}$ is
  non-decreasing with $i$, and the last inequality uses our assumption that
  $x_{t}^{\T}\nabla f\left(x_{t}\right)\leq0$ along with
  $\lambda^{\left(d\right)}\left(A\right)\leq \beta$.  
\end{proof}

\subsection{Proof of Lemma \ref{lem:monotone-weak}}
\label{sec:finally-proof-monotone-weak}

\restate{\lemMonoWeak*}

  \newcommand{\Rlow}{R_{\rm low}}
  \newcommand{\hessremain}{\Delta}
  By definition of the gradient descent iteration, we have
  \begin{equation}
    \label{eqn:norm-one-step}
    \norm{ x_{t+1}}^{2}=\norm{ x_{t}}^{2}-2\eta x_{t}^{\T}\grad f\left(x_{t}\right)+\eta^{2}\norm{ \grad f\left(x_{t}\right)}^{2},
  \end{equation}
  and therefore if we can show that $x_{t}^{\T}\grad f\left(x_{t}\right)\leq0$
  for all $t$, the lemma holds.  We give a proof by induction.  The basis of
  the induction $x_{0}^{\T}\grad f\left(x_{0}\right)\leq0$ is immediate as
  $r \mapsto f(-r b / \norm{b})$ is decreasing until
  $r = R_c$
  (recall the definition~\eqref{eq:Rc-def}), and $x_0^T \nabla f(x_0) = 0$ for
  $r \in \{0, R_c\}$.  Our induction assumption is that
  $x_{t'-1}^{\T}\grad f\left(x_{t'-1}\right)\leq0$ (and hence also
  $\norm{ x_{t'}} \geq\norm{ x_{t'-1}} $) for  $t'\leq t$ and we wish to show
  that $x_{t}^{\T}\grad f\left(x_{t}\right)\leq0$.
  Note that
  \begin{equation*}
  x^{\T}\grad f\left(x\right)=x^{\T}Ax+\rho\norm{ x}^{3}+b^{\T}x\geq\rho\norm{ x}^{3}-\gamma\norm{ x}^{2}-\norm{ b} \norm{ x} 
  \end{equation*}
  and therefore $x^{\T}\grad f\left(x\right)>0$ for every
  $\norm{ x} > \Rlow\defeq
  \frac{\gamma}{2\rho}+\sqrt{\left(\frac{\gamma}{2\rho}\right)^{2} +
    \frac{\norm{b} }{\rho}}$.
  Therefore, our induction assumption also implies
  $\norm{ x_{t'-1}} \leq \Rlow\leq R$ for every $t'\leq t$.
  
  Using that $\nabla^2 f$ is $2\rho$-Lipschitz,
  a Taylor
  expansion immediately
  implies~\cite[Lemma~1]{NesterovPo06} that for all vectors $\Delta$, we have
  \begin{equation}
    \norm{\nabla f(x + \Delta)
      - (\nabla f(x) + \nabla^2 f(x) \Delta)} \le
    \rho \norm{\Delta}^2.
    \label{eqn:lip-hess-bound}
  \end{equation}
  Thus, if we define 
  $\hessremain_t \defeq \frac{1}{\eta^2} [\nabla f(x_t) - (\nabla f(x_{t-1}) -
  \eta \nabla^2 f(x_{t-1}) \nabla f(x_{t-1}))]$, we have
  $\norm{\hessremain_t} \le \rho \norm{\nabla f(x_{t-1})}^2$, and using
  the iteration $x_t = x_{t-  1} - \eta \nabla f(x_{t-1})$ yields
  \begin{align}
    x_{t}^{\T}\grad f\left(x_{t}\right)
    & =x_{t-1}^{\T}\grad f\left(x_{t-1}\right)
    -\eta\norm{ \grad f\left(x_{t-1}\right)}^2
    - \eta \underbrace{x_{t-1}^{\T}\grad^{2}f\left(x_{t-1}\right)
      \grad f\left(x_{t-1}\right)}_{\defeq \mc{T}_1} \nonumber \\
    & \qquad+\eta^{2}
    \underbrace{\grad f\left(x_{t-1}\right)^T 
    \grad^{2}f\left(x_{t-1}\right)\grad f\left(x_{t-1}\right)}_{\defeq \mc{T}_2}
    +\eta^{2} \underbrace{x_{t}^{\T}\hessremain_t}_{\defeq \mc{T}_3}.
    \label{eqn:three-fun-terms}
  \end{align}
  We bound each of the terms $\mc{T}_i$ in turn. We have that
   \begin{alignat*}{1}
      \mc{T}_1 &= x_{t-1}^{\T}\grad^{2}f\left(x_{t-1}\right)\grad 
  f\left(x_{t-1}\right)=x_{t-1}^{\T}A\grad f\left(x_{t-1}\right)+2\rho\norm{ 
  x_{t-1}} 
  x_{t-1}^{\T}\grad f\left(x_{t-1}\right) \\
  & \qquad \ge (\beta + 2\rho \norm{x_{t-1}})x_{t-1}^{\T}\grad 
  f(x_{t-1}) \ge (\beta + 2\rho R)x_{t-1}^{\T}\grad 
  f(x_{t-1}),
      \label{eqn:x-transpose-A-expansion}
    \end{alignat*}
    where both inequalities follow from the induction assumption; the first is 
    Lemma~\ref{lemma:x-transpose-A-nice} and the second is due to 
    $\norm{x_{t-1}} \le R$ and $x_{t-1}^{\T}\grad 
    f(x_{t-1}) \le 0 $.
    
  Treating the second order term $\mc{T}_2$, we obtain that
  \begin{alignat*}{1}
    \mc{T}_2  \le \opnorm{\hess f(x_{t-1})} 
    \norm{\grad f(x_{t-1})}^2 \leq\left(\beta+2\rho R\right)
    \norm{ \grad f\left(x_{t-1}\right)}^{2},
  \end{alignat*}
  and, by the Lipschitz bound~\eqref{eqn:lip-hess-bound},
  the remainder term $\mc{T}_3$ satisfies
  \begin{align*}
    \mc{T}_3 = x_{t}^{\T} \hessremain_t &
    \leq\norm{ x_{t}} \norm{ r} \leq\rho\norm{ x_{t}}
      \norm{ \grad f\left(x_{t-1}\right)}^{2}
     \leq \rho\norm{ x_{t-1}-\eta\grad f\left(x_{t-1}\right)} \norm{ \grad 
     f\left(x_{t-1}\right)}^{2}\\
    & \leq\rho\norm{ x_{t-1}} \norm{ \grad f\left(x_{t-1}\right)}^{2}+\rho\eta\norm{ \grad f\left(x_{t-1}\right)}^{3}.
  \end{align*}
  Using that
  $\norm{\nabla f(x)} = \norm{\nabla f(x) - \nabla f(\s)} \le (\beta + 2R)
  \norm{x - \s} \le R(\beta + 2 \rho R)$
  for $\norm{x} \le R$ and that $\eta\leq1/2\left(\beta+2\rho R\right)$, our
  inductive assumption that $\norm{x_{t-1}} \le R$ thus guarantees that
  $\mc{T}_3 \le 2 \rho R \norm{\nabla f(x_{t-1})}^2$.  Combining our bounds 
  on
  the terms $\mc{T}_i$ in expression~\eqref{eqn:three-fun-terms}, we have that
  \begin{equation*}
    x_{t}^{\T}\grad f\left(x_{t}\right)
    \leq\left(1-\eta\left(\beta +2\rho R\right)\right)x_{t-1}^{\T}\grad f\left(x_{t-1}\right)
    -\left(\eta - \eta^{2}(\beta+4\rho R)\right) \norm{ \grad f(x_{t-1})}^{2}.
  \end{equation*}
  Using $\eta\leq1/\left(\beta+4\rho R\right)$ shows that
  $x_{t}^{\T}\grad f\left(x_{t}\right)\leq0$, completing our induction.
  By the expansion~\eqref{eqn:norm-one-step}, we have
  $\norm{x_t} \le \norm{x_{t+1}}$ as desired, and that
  $x_t^T \grad f(x_t) \le 0$ for all $t$ guarantees that $\norm{x_t} \le \Rlow \le R$.
\section{Proofs of technical results from Section~\ref{sec:proofs}}
\label{app:proof}
As in the statement of our major theorems and as we note in
the beginning of Section~\ref{sec:proofs}, we tacitly assume
Assumptions~\ref{assu:step-size} and~\ref{assu:init} throughout this section. 

\subsection{Proof of Lemma~\ref{lem:lin-converge}}
\label{sec:proof-lin-converge}

\restate{\lemLinConv*}

Expanding $x_t = x_{t - 1} - \eta \nabla f(x_{t-1})$, we have
\begin{equation}
  \label{eqn:lyapunov-distance}
  \norm{x_{t}-\s}^2
  =\norm{x_{t-1}-\s}^2 - 2\eta\left(x_{t-1}-\s\right)^{\T} \grad f(x_{t-1})
  + \eta^2\norm{\grad f\left(x_{t-1}\right)}^2.
\end{equation}
Using the equality
$\grad f\left(x\right)=\As\left(x-\s\right)-\rho\left(\norm{\s} -\norm{x}
\right)x$, we rewrite the cross-term  $\left(x_{t-1}-\s\right)^{\T}
\grad
f\left(x_{t-1}\right)$ as
\begin{flalign}
  \nonumber
  \left(x_{t-1}-\s\right)^{\T}&\As 
  \left(x_{t-1}-\s\right)+\rho\left(\norm{x_{t-1}}
  -\norm{\s} \right)(\norm{x_{t-1}}^2-\s^{\T}x_{t-1}) \\ \nonumber
  = & 
  \left(x_{t-1}-\s\right)^{\T}\left(\As+\frac{\rho}{2}\left(\norm{x_{t-1}}
    -\norm{\s} \right) I \right) \left(x_{t-1}-\s\right) \\
   & +\frac{\rho}{2}\left(\norm{\s} -\norm{x_{t-1}} 
   \right)^2\left(\norm{x_{t-1}} 
    +\norm{\s} \right).
    \label{eqn:cross-term-lyapunov}
\end{flalign}
Moving to the second order term $\norm{\grad f\left(x_{t-1}\right)}^2$
from the expansion~\eqref{eqn:lyapunov-distance},
we find
\begin{alignat*}{1}
  \norm{\grad f\left(x_{t-1}\right)}^2 & 
  =\norm{\As\left(x_{t-1}-\s\right)+\rho\left(\norm{x_{t-1}} -\norm{\s} 
  \right)x_{t-1}}^2\\
  & 
  \leq2\left(x_{t-1}-\s\right)^{\T}\As^2\left(x_{t-1}-\s\right)+2\rho^2\left(\norm{x_{t-1}}
   -\norm{\s} \right)^2\norm{x_{t-1}}^2.
\end{alignat*}
Combining this inequality with the cross-term
calculation~\eqref{eqn:cross-term-lyapunov}
and the squared distance~\eqref{eqn:lyapunov-distance} we obtain
\begin{alignat*}{1}
  \norm{x_{t}-\s}^2 & \le
  (x_{t-1}-\s)^{\T}(I-2\eta \As(I-\eta \As)-\eta\rho(\norm{x_{t-1}} 
  -\norm{\s} ) I )
  (x_{t-1}-\s) \\
  & \qquad-\eta\rho\left(\norm{\s} -\norm{x_{t-1}} \right)^2\left(\norm{x_{t-1}} \left(1-2\eta\rho\norm{x_{t-1}} \right)+\norm{\s} \right).
\end{alignat*}
Using $\eta\leq\frac{1}{4\left(\beta+\rho 
R\right)}\leq\frac{1}{4\opnorm{\As}}$ yields 
$2\eta \As\left(1-\eta \As\right)\succeq\frac{3}{2}\eta
\As\succeq\frac{3}{2}\eta \left(-\gamma+\rho\norm{\s} \right)I$, so
\begin{align*}
  \norm{x_{t}-\s}^2 & \leq
  \left(1-\frac{\eta}{2}\left[-3\gamma+\rho\left(\norm{\s} +2\norm{x_{t-1}} \right)\right]\right)\norm{x_{t-1}-\s}^2 \\
  & \qquad ~ -\eta\rho\left(\norm{\s} -\norm{x_{t-1}} \right)^2\norm{\s}.
\end{align*}

\subsection{Proof of Lemma~\ref{lem:growth}}
\label{sec:proof-growth}

\restate{\lemGrowth*}
 
The claim is trivial when $\gamma\le0$ as it clearly implies $\rho \norm{x_t} 
\ge \gamma$, so we assume $\gammaplus = \gamma > 0$. Using
Proposition~\ref{prop:converge} that gradient descent is
convergent, we may define $t\opt = \max\{t : \rho \norm{x_t} \le \gamma 
- 
\smallval\}$.
Then for every $t\leq t\opt$, the gradient descent 
iteration~\eqref{eq:grad-iter} 
satisfies
\begin{align*}
  \frac{x_{t}^{\left(1\right)}}{-\eta b^{\left(1\right)}}
  & =\left(1+\eta\gamma-\eta\rho\norm{x_{t-1}} 
  \right)\frac{x_{t-1}^{\left(1\right)}}{-\eta b^{\left(1\right)}}+1 \\ & 
  \geq\left(1+\eta\smallval\right)\frac{x_{t-1}^{\left(1\right)}}{-\eta 
  b^{\left(1\right)}}+1
  \geq \cdots 
  \geq\frac{1}{\eta\smallval}\left(\left(1+\eta\smallval\right)^{t}-1\right).
\end{align*}
Multiplying both sides of the equality by $\eta |b\supind{1}|$ and
using that $x_t\supind{1} b\supind{1} \le 0$, we have
\begin{equation*}
  \frac{\gamma-\smallval}{\rho}\geq\norm{x_{t\opt}} \geq 
  |x_{t\opt}^{\left(1\right)}|\geq\frac{|b^{\left(1\right)}|}{\smallval}\left(\left(1+\eta\smallval\right)^{t\opt}-1\right).
\end{equation*}
Consequently,
\begin{equation*}
  t\opt\leq\frac{\log\left(1+\frac{(\gamma-\smallval)\smallval}{\rho|b\supind{1}|}\right)}{
    \log(1+\eta\smallval)}
  \leq\frac{2}{\eta\smallval}\log\left(1+\frac{\gammaplus^2}{4\rho|b^{\left(1\right)}|}\right),
\end{equation*}
where we used
$\eta\smallval \leq \eta\gamma \leq \gamma/\beta \leq1$, whence
$\log(1 + \eta \smallval) \ge \frac{\eta \smallval}{2}$, and
$ \gamma \smallval - \smallval^2 \le \sup_{x\ge0}\{x(\gamma-x)\} \le 
\frac{ 
\gammaplus^2 }{ 4} $.

\subsection{Proof of Lemma~\ref{lem:proj-lin-converge}}
\label{sec:proof-proj-lin-converge}

\restate{\lemProjLinConverge*}

For typographical convenience, we prove the result with $t + 1$ replacing $t$.
Using the commutativity of $\Pi$ and $A$, we have $\Pi \As=\As\Pi$, and therefore also  $\Pi \As^{1/2}=\As^{1/2}\Pi$, implying
\begin{flalign}
  \nonumber \norm{\Pi \As^{1/2}\left(x_{t + 1}-\s\right)}^2
  = & \norm{\Pi
    \As^{1/2}\left(x_t-\s\right)}^2\\ &-2\eta\left(x_t-\s\right)^T \!\!
  \As\Pi
  \grad f\left(x_t\right)+\eta^2\norm{\Pi \As^{1/2}\grad
    f\left(x_t\right)}^2.
  \label{eqn:proj-lyapunov}
\end{flalign}
We substitute
$\grad f\left(x\right)=\As\left(x-\s\right)-\rho\left(\norm{\s} -\norm{x}
\right)x$ in the cross term to obtain
\begin{align*}
  \lefteqn{\left(x_t-\s\right)^{\T} \Pi \As\grad f\left(x_t\right)} \\
  & \qquad ~ = \left(x_t-\s\right)^{\T}\Pi \As^2\Pi\left(x_t-\s\right)
  - \rho\left(\norm{\s} -\norm{x_t} \right)x_t^{\T}\Pi \As\left(x_t-\s\right).
\end{align*}
Substituting
$\As\left(x-\s\right)=\grad f\left(x\right)+\rho\left(\norm{\s} -\norm{x}
\right)x$ in the last term yields
\begin{equation}
  x_t^{\T}\Pi \As\left(x_t-\s\right)=x_t^{\T}\Pi\grad
  f\left(x_t\right)+\rho\left(\norm{\s} -\norm{x_t} \right)\norm{\Pi
    x_t}^2.
  \label{eqn:proj-cross-term}
\end{equation}
Invoking Lemma~\ref{lemma:monotone} and the fact that
$x_t^{\T}\grad f\left(x_t\right)\leq0$, we get
\begin{alignat*}{1}
  x_t^{\T}\Pi\grad f\left(x_t\right) & =x_t^{\T}\grad f\left(x_t\right)-x_t^{\T}\left(I-\Pi\right)\grad f\left(x_t\right)\\
  & \leq-x_t^{\T}\left(I-\Pi\right)\grad f\left(x_t\right)\\
  & =-x_t^{\T}\left(I-\Pi\right)\As\left(x_t-\s\right)+\rho\left(\norm{\s} 
  -\norm{x_t} \right)\norm{\left(I-\Pi\right)x_t}^2\\
  & \leq\opnorm{\left(I-\Pi\right)\As} \norm{x_t} \norm{x_t-\s} 
  +\rho\left(\norm{\s} -\norm{x_t} \right)\norm{\left(I-\Pi\right)x_t}^2\\
  & \leq\sqrt{2}\opnorm{\left(I-\Pi\right)\As} 
  \norm{\s}^2+\rho\left(\norm{\s} -\norm{x_t} 
  \right)\norm{\left(I-\Pi\right)x_t}^2,
\end{alignat*}
where in the last line we used $x_t^{\T}\s \ge 0$ (by
Lemma~\ref{lemma:signs}). Combining this with
the cross terms~\eqref{eqn:proj-cross-term}, we find that
\begin{subequations}
\begin{alignat}{1}\label{eq:subspace-bound-1}
  x_t^{\T}\Pi \As\left(x_t-\s\right) & \leq\sqrt{2}
  \opnorm{\left(I-\Pi\right)\As} \norm{\s}^2
  +\rho\left(\norm{\s} -\norm{x_t} \right)\norm{x_t}^2.
\end{alignat}
Moving on to the second order term in the expansion~\eqref{eqn:proj-lyapunov},
we have
\begin{alignat}{1}
  \norm{\Pi \As^{1/2}\grad f\left(x_t\right)}^2 & =\norm{\Pi 
  \As^{3/2}\left(x_t-\s\right)+\rho\left(\norm{x_t} -\norm{\s} 
  \right)\As^{1/2}\Pi x_t}^2 \nonumber\\
  & \leq2\norm{\Pi \As^{3/2}\left(x_t-\s\right)}^2+2\rho^2\opnorm{\Pi 
  \As} 
  \left(\norm{x_t} -\norm{\s} \right)^2\norm{x_t}^2. 
  \label{eq:subspace-bound-2}
\end{alignat}
\end{subequations}
Substituting the bounds~\eqref{eq:subspace-bound-1} 
and~\eqref{eq:subspace-bound-2} into the
expansion~\eqref{eqn:proj-lyapunov},
we have
\begin{align*}
 \norm{\Pi \As^{1/2}\left(x_{t + 1}-\s\right)}^2 
\leq & \left(x_t-\s\right)^{\T}\left(I-2\eta\Pi \As\left(I-\eta\Pi 
\As\right)\right)\Pi 
\As\left(x_t-\s\right)\\
&  +2\eta\rho\left(\norm{\s} -\norm{x_t} 
  \right)\left[\sqrt{2}\opnorm{\left(I-\Pi\right)\As}\norm{\s} 
  ^2\right. \\ & \qquad\qquad
  + \left. \left(1+\eta\opnorm{\Pi \As} \right)\rho\left(\norm{x_t} 
  -\norm{\s} \right)\norm{x_t}^2\right].
\end{align*}
Using $\eta\leq1/(4\left(\beta+\rho R\right))$, which guarantees
$0 \preceq \eta\Pi \As\preceq I/4 \prec I/2$, together with the 
assumption that $\Pi \As\succeq\smallval\Pi$ gives
\begin{equation*}
0 \preceq I-2\eta\Pi \As\left(I-\eta\Pi 
\As\right) \preceq (1-\eta\smallval)I
\end{equation*} 
and therefore
\begin{alignat*}{1}
&\norm{\Pi \As^{1/2}\left(x_{t + 1}-\s\right)}^2  
\leq\left(1-\eta\smallval\right)\norm{\Pi 
	\As^{1/2}\left(x_t-\s\right)}^2\\
& \quad\quad\quad\quad+\sqrt{8}\eta\rho\left(\norm{\s} -\norm{x_t} 
\right)\left[\rho\left(\norm{\s} -\norm{x_t} \right)\norm{x_t}^2
+\opnorm{\left(I-\Pi\right)\As}\norm{\s}^2\right].
\end{alignat*}

\subsection{Proof of Lemma~\ref{lemma:eventually-proj-lin-converge}}
\label{sec:proof-eventually-proj-lin-converge}

\restate{\lemEventuallyProjLinConverge*}

The conditions of the lemma imply that for $\tau\geq0$,
\begin{equation*}
  \rho(\norm{\s} -\norm{x_{t+\tau}}) \leq 4\sqrt{\smallval\bar{\smallval}}/3
\end{equation*}
and also that
$\opnorm{\left(I-\Pi_{\smallval}\right)\As}\leq2\smallval \le 
2\sqrt{\smallval\bar{\smallval}}$ 
(Eq.~\eqref{eq:subpsace-opnorm-bound}), 
and
$\Pi_{\smallval}\As\succeq\bar{\smallval}I$. Substituting these bounds 
into
Lemma~\ref{lem:proj-lin-converge} along with $\norm{x_{t-1}} \leq\norm{\s} $
(Lemma~\ref{lemma:monotone}), we get
\begin{alignat*}{1}
  \norm{\Pi_{\smallval}\As^{1/2}\left(x_{t+\tau}-\s\right)}^2 & \leq 
  \left(1-\eta\bar{\smallval}\right) 
  \norm{\Pi_{\smallval}\As^{1/2}\left(x_{t+\tau-1}-\s\right)}^2+ 
  13\eta\smallval\bar{\smallval}\norm{\s}^2.
\end{alignat*}
Iterating this $\tau$ times gives
\begin{alignat*}{1}
  \norm{\Pi_{\smallval}\As^{1/2}\left(x_{t+\tau}-\s\right)}^2 & 
  \leq\left(1-\eta\bar{\smallval}\right)^{\tau}\norm{\Pi_{\smallval}\As^{1/2}\left(x_{t}-\s\right)}^2+13\smallval\norm{\s}^2\left(1-\left(1-\eta\bar{\smallval}\right)^{\tau}\right)\\
 & \leq2\left(\beta+\rho\norm{\s} 
 \right)\norm{\s}^2e^{-\eta\bar{\smallval}\tau}+13\norm{\s}^2\smallval
\end{alignat*}
where the last transition uses that
\begin{equation*}
  \norm{\Pi_{\smallval}\As^{1/2}\left(x_{t}-\s\right)}^2\leq\opnorm{\As} 
  \norm{x_{t}-\s}^2\leq\left(\beta+\rho\norm{\s} \right)2\norm{\s}^2.
\end{equation*}

\subsection{Proof of Lemma~\ref{lem:rand-prop}}
\label{sec:proof-rand-prop}
\restate{\lemPertProp*}

To establish part (i) of the lemma, note that marginally
$[q^{\left(1\right)}]^2\sim\mathrm{Beta}(\frac{1}{2},\frac{d-1}{2})$ and
that $q^{\left(1\right)}$ is symmetrically distributed. Therefore, for $d>2$
the density of
$\tilde{b}^{\left(1\right)}=b^{\left(1\right)}+\sigma q^{\left(1\right)}$ is
maximal at $b^{\left(1\right)}$ and is monotonically decreasing in the
distance from $b^{\left(1\right)}$. Therefore we have
\begin{align*}
  \P\left(|\tilde{b}\supind{1}| \le \sigma\sqrt{\pi}\delta/\sqrt{2 d}\right) \le  
  \P\left(|q\supind{1}| \le \sqrt{\pi}\delta / \sqrt{2 d} \right) \le \delta,
\end{align*}
where the bound $p_1(u) \le \sqrt{d / (2\pi u)}$
on the density $p_1$ of $q\supind{1}$ yields the last inequality.

Part (ii) of the lemma is immediate, as
\begin{equation*}
  |f\left(x\right)-\tilde{f}\left(x\right) | =
  |(b-\tilde{b})^{\T}x|\leq\sigma\norm{q} \norm{
    x} =\sigma \norm{x}.
\end{equation*}

To show part (iii) of the lemma, we first note that $\norm{\s}$ is
a well-defined function of $b$, because $\s$ is not unique only when
$\norm{\s} =\gamma/\rho$ (see Proposition~\ref{prop:classic-char}).  
Next, from the relation $b=-(A+\rs I)\s$
we see that the inverse mapping $\s \mapsto b$ is a smooth function, with Jacobian
\begin{equation*}
  \frac{\del b}{\del \s}=-\left(A+\rs I +\rho\frac{\s\s^{\T}}{\norm{\s}}\right).
\end{equation*}
Let us now evaluate $\del\norm{\s}/\del b$ when
the mapping $\s \mapsto b(\s) = -(A + \rho \norm{\s} I) \s$
is invertible (\ie in the case that $\norm{\s} >\gamma/\rho$); 
the inverse function theorem yields
\begin{equation*}
  \frac{\del\norm{\s}}{\del b}=\frac{1}{2\norm{\s}}\frac{\del\left(\s^{\T}\s\right)}{\del b}=\frac{\del \s}{\del b}\cdot \frac{\s}{\norm{\s}}=-\left(A+\rs I +\rho\frac{\s\s^{\T}}{\norm{\s}}\right)^{-1}\frac{\s}{\norm{\s}}.
\end{equation*}

\newcommand{\bfunc}[1]{b^{#1}}
\newcommand{\sfunc}[1]{\s^{#1}}

For $\theta\in[0,1]$, let $\bfunc{\theta} \defeq b+\theta(\tilde{b}-b)$, let $\sfunc{\theta}$ denote a global minimizer of $f$ with $b$ replaced with $\bfunc{\theta}$. Let $\tau = \argmin_{\theta\in[0,1]}{\norm{\sfunc{\theta}}}$ and let $H(\theta)=\norm{\sfunc{\theta}}-\norm{\sfunc{\tau}}$. Using the chain rule, the calculations above, and $\norms{\tilde{b}-b}=\sigma$, we have
\begin{equation*}
  H'(\theta) = {(\tilde{b}-b)^T \frac{\del\norm{\sfunc{\theta}}}{\del b}} \le \sigma \norm{\left(A+\rho \norm{\sfunc{\theta}} I +\rho\frac{\sfunc{\theta}(\sfunc{\theta})^{\T}}{\norm{\sfunc{\theta}}}\right)^{-1} \frac{\sfunc{\theta}}{\norm{\sfunc{\theta}}} } \le \frac{\sigma}{\rho H(\theta)},
\end{equation*}
where the final inequality follows from the fact that $A+\rho \norm{\sfunc{\tau}}I\succeq 0$ (since $\sfunc{\tau}$ is a global minimizer of a cubic-regularized objective) and therefore 
\[
  M=A+\rho \norm{\sfunc{\theta}} I+\rho\frac{\sfunc{\theta}(\sfunc{\theta})^{\T}}{\norm{\sfunc{\theta}}} =
  A+\rho \norm{\sfunc{\tau}} I+\rho\frac{\sfunc{\theta}(\sfunc{\theta})^{\T}}{\norm{\sfunc{\theta}}} + \rho\left(\norms{\sfunc{\theta}} - \norm{\sfunc{\tau}}\right)I \succeq \rho H(\theta)I \succeq 0
\]
(by definition of $\tau$), meaning that all eigenvalues of $M^{-1}$ are at most $1/(\rho H(\theta))$. 
We conclude that
\begin{equation*}
 (H^2(\theta))' = 2H(\theta)H'(\theta) \le \frac{2\sigma}{\rho}.
\end{equation*}
Integrating from $\tau$ to $1$ (and recalling that $H(\tau)=0$), we obtain
\begin{equation*}
  (\norm{\stilde}-\norm{\sfunc{\tau}})^2 = H^2(1) \le \frac{2\sigma}{\rho}(1-\tau)\le \frac{2\sigma}{\rho}
\end{equation*}
and consequently (by definition of $\tau$)
\begin{equation*}
  \norm{\stilde}-\norm{\s}\le \norm{\stilde}-\norm{\sfunc{\tau}} \le \sqrt{\frac{2\sigma}{\rho}}.
\end{equation*}
The same upper bound on $\norm{\s}-\norm{\stilde}$ follows analogously by integrating $(H^2(\theta))'$ from $\tau$ to $0$.

\section{Proof of Proposition~\ref{prop:linesearch}}
\label{app:linesearch}
We begin with a lemma
 implicitly assuming the conditions of
 Proposition~\ref{prop:linesearch}.
\begin{lemma}
  \label{lemma:small-x-line-search}
  For all $t$ we have $\norm{x_t} \le 2R$, with $R$ given by~\eqref{eq:R-def}.
\end{lemma}
\begin{proof}
  Note that $R$ minimizes the
  polynomial $-\|b\|r-\beta r^2/2 + \rho r^3/3$ as it solves
  $-\norm{b} - \beta R + \rho R^2 = 0$. This implies that for every
  $\|x\| > 2R$ we have
  \begin{align*}
    f(x)  \ge -\|b\|\|x\|-\frac{\beta}{2} \|x\|^2 + \frac{\rho}{3} \|x\|^3
    > 2R \left(  -\|b\|  - \beta R + \frac{4\rho}{3} R^2 \right) 
    = \frac{2\rho}{3} R^3 \ge 0,
  \end{align*}
  where the first inequality follows because $b^{\T}x \ge -\|b\|\|x\|$ and
  $\beta \ge \opnorm{A}$, the second because
  $-\|b\|\|x\|-\beta \|x\|^2/2 + \rho \|x\|^3/3$ is increasing in $\|x\|$ for
  $\|x\| \ge R$, and in the last inequality we substituted
  $\|b\| = \rho R^2 - \beta R$. By Assumption~\ref{assu:init}, $f(x_0)\le 0$,
  and the definition~\eqref{eq:constrained-steepest} of the step size $\eta_t$
  guarantees that $f(x_t)$ is non-increasing. Thus $f(x_t) \le 0$ for all
  $t$, so $\|x_t\|\le2R$.
\end{proof}

\newcommand{\etasimple}{\eta_{\rm feas}}
As in our proof of Lemma~\ref{lemma:signs}, we focus on the on the first
coordinate of the iteration~\eqref{eq:grad-variablestep} (\ie 
$x_{t + 1} = x_t - \eta_t \nabla f(x_t)$) in the eigenbasis of $A$, writing
\begin{equation*}
  x_{t+1}^{\left(1\right)}=\left(1-\eta_t \left[-\gamma+\rho\| x_{t}\|
    \right]\right)
  x_{t}^{\supind{1}} - \eta_t b\supind{1}.
  \label{eq:coord-rec-linesearch}
\end{equation*}
By the constrains in the definition~\eqref{eq:constrained-steepest} of the 
step size $\eta_t$, we have
\begin{equation*}
  1-\eta_t (-\gamma+\rho\| x_{t}\| ) \ge 1-\eta_t
  \hinge{\frac{\grad f(x_t)^{\T} A \grad f(x_t)}{\|\grad f(x_t)\|^2}+\rho \norm{x_{t}} }
  \ge 0.
\end{equation*}
By Assumption~\ref{assu:init}, $b^{(1)}x_0^{(1)} \le 0$, so
$b^{(1)}x_t^{(1)} \le 0$ for every $t$.
Since $u^\T A u / \|u\|^2 \le \opnorm{A} \le \beta$ for all $u$ and
$\|x_t\| \le 2R$ for every $t$, the step size $\etasimple \defeq 1/(\beta+4\rho 
R)$
is always feasible, and we have
$f(x_{t+1}) \le f(x_t - \etasimple \grad f(x_t))$. Moreover, since $f$ is
$\beta+4\rho R$-smooth on the set
$\mathbb{B}_{2R} = \{x \in \R^d : \norm{x} \le 2R\}$, and as
$x_t \in \mathbb{B}_{2R}$ for all $t$ by
Lemma~\ref{lemma:small-x-line-search}, we have $f(x_{t+1}) \le f(x_t) -
\frac{\etasimple}{2} \|\grad f(x_t)\|^2$, which implies $\grad f(x_t) \to 0$. 
Having established $b^{(1)}x_t^{(1)} \le 0$ for every $t$ and $\grad f(x_t) \to 
0$ as $t\to\infty$, the remainder of the proof is identical to that of
Proposition~\ref{prop:converge}. 

\section{Proofs from Section \ref{sec:trustregion}}\label{app:trustregion}

\subsection{Proof of Lemma~\ref{lem:SSP}}
\label{sec:proof-SSP}
\restate{\lemSSP*}

For $x_0$ defined in line \ref{line:x0} of Alg.~\ref{alg:SSP}, we have
$f(x_0) = - (1/2)R_c \|b\| - (\rho/6)R_c^3$, where
$R_c$ is the Cauchy radius~\eqref{eq:Rc-def}.
Therefore a sufficient condition for
$f(x_0) \le -(1-\eps')\rho\slb^3/6$ is $R_c\|b\| \ge \rho\slb^3/3$. We have
\begin{equation*}
  R_c \ge 
  \frac{-\beta}{2\rho}+\sqrt{\left(\frac{\beta}{2\rho}\right)^{2}+\frac{\| 
  b\|}{\rho}} \ge \min\left\{\frac{2\|b\|}{3\beta} , 
  \sqrt{\frac{\|b\|}{3\rho}}\right\} \ge 
  \frac{1}{3}\min\left\{\frac{\|b\|}{\beta} , \sqrt{\frac{\|b\|}{\rho}}\right\},
\end{equation*}
where the second inequality follows from
$\sqrt{1+\alpha} \ge 1 + (\min\{\alpha/3 , \sqrt{\alpha/3}\})$ for every
$\alpha\ge0$.  Thus, Alg.~\ref{alg:SSP} returns $x_0$ whenever 
$\|b\|\min\big\{\frac{\|b\|}{\beta} , \sqrt{\frac{\|b\|}{\rho}}\big\} \ge 
\rho\slb^3$,
which is equivalent to the second part of the ``or'' condition in the lemma.

Now, suppose that the algorithm does not return $x_0$, \ie$f(x_0) > 
-(1-\eps')\rho\slb^3/6$. Since $f(x_0) < - (\rho/6)R_c^3$ , this implies 
that $R_c < \slb$.  Since $\rho R_c \ge \rho R - \beta$, with $R$ defined in 
\eqref{eq:R-def}, we have $\rho \slb > \rho R - \beta$. Therefore, $\eta \le 
1/(8\beta + 4\rho\slb) \le 1/(4\beta + 4\rho R)$, and so the
stepsize $\eta$ required in the lemma statement satisfies 
Assumption~\ref{assu:step-size}.
Since we choose $\tilde{x}_0$ in accordance with Assumption 
\ref{assu:init}, we may invoke Theorem~\ref{thm:pert-main-result} with 
$\eps = \rho \norm{\s}^3 \eps'/12$.

Our setting $\sigma = \frac{\rho^3 \slb^4 (\eps')^2}{2(120\beta+240\rho\slb)^2} = 
\frac{\rho \slb^4 \eps}{200\norm{\s}^4(\beta+2\rho\slb)^2}$ implies 
\begin{equation*}
  \sigbar = \frac{200(\beta+2\rs)^2 \nss}{\rho\eps^2} \sigma = 
  \left(\frac{\beta+2\rs}{\beta+2\rho\slb}\right)^2 \cdot \frac{\slb^4}{\norm{\s}^4}~.%
\end{equation*}
Therefore, assuming $\slb \le \norm{\s}$, we have that $(\slb/\norm{\s})^4 \le 
\sigbar \le (r/\ns)^2 
\le 1$.
Substituting these upper and lower bounds,
Theorem~\ref{thm:pert-main-result} 
shows that, with probability at least $1-\delta$, $f(\tilde{x}_t) \le f(\s) + 
(1+\sigbar)\eps \le f(\s) + 2\eps$ for all
\begin{equation*}
  t \ge \frac{40 \nss}{\eta\eps}\left( 6\log\left(1+\frac{50\sqrt{d}}{\delta}\cdot 
      \frac{\norm{\s}^4}{\slb^4}\right)
    +20\log\left(\frac{(\beta+2\rs)\nss}{\eps}\right)\right) \defeq 
  \tilde{T}_\eps^{\mathrm{sub}}.
\end{equation*}
Using  $\rho \ns \le \beta + \rho R \le \frac{1}{4\eta}$ and plugging in $\eps = 
\rho\norm{\s}^3\eps'/12$, we see that
\begin{flalign*}
  \tilde{T}_\eps^{\mathrm{sub}}& \le \frac{480}{\eta\norm{\s}\eps'}\left( 
    6\log\left(1+\frac{\sqrt{d}}{\delta}\right)
    +6\log\left(\left[\frac{1}{\eta\rho\slb}\right]^4\right)
    +20\log\left(\frac{6}{\eta\rho\norm{\s}\eps'}\right)\right)\\
  &\le \frac{480}{\eta\slb\eps'}\left( 
    6\log\left(1+\frac{\sqrt{d}}{\delta}\right)
    +44\log\left(\frac{6}{\eta\rho\slb\eps'}\right)\right),
\end{flalign*}
where we used $\slb\le\norm{\s}$ and $\eps' < 1$. 
Therefore $T$ defined in line \ref{line:T} is larger 
than $\tilde{T}_\eps^{\mathrm{sub}}$, so with probability at least 
$1-\delta$, there exists $t\le T$ for which 
$f(\tilde{x}_t) \le f(\s) + \rho\norm{\s}^3\eps'/6$. Recalling that $f(\s) \le 
-\rho\|\s\|^3/6 \le 
-\rho\slb^3/6$ by the bound~\eqref{eq:fs-lower-bound} completes the proof.

\subsection{Proof of Proposition~\ref{prop:SP}}
\label{sec:proof-SP}

\restate{\propSP*}

We always call \callSSP{} with $\eps' = 1/2$ and
$\slb = \sqrt{\epsilon/(9\rho)}$. As $\epsilon \le \beta^2/\rho$ we have that
$\eta = 1/(10\beta) \le 1/(8\beta + 4\rho\slb)$. Since
$\opnorm{\hess g(x)} \le\beta$ by Assumption~\ref{assu:g}, we conclude that
Lemma \ref{lem:SSP} applies to each call of \callSSP{}. Note that by 
construction of Alg.~\ref{alg:SP}, every call to~\callSSP{}---except the last 
one---reduces the value of $g$ by at least $\Kprog \epsilon^{3/2} 
\rho^{-1/2}$ (Line~\ref{line:progress}). Therefore, by a standard progress 
argument, the algorithm calls \callSSP{}
at most 
\begin{equation}\label{eq:Kmax-def}
\Kmax = 1 + \ceil{\frac{\sqrt{\rho}(g(y_0)-\glb)}{\Kprog \epsilon^{3/2}}} 
\le O(1)\cdot \frac{\sqrt{\rho}(g(y_0)-\glb)}{\epsilon^{3/2}}
\end{equation}
times, where we used $\epsilon \le 
\rho^{1/3}(g(y_0)-\glb)^{2/3}$. 
Letting $\mc{E}$ be the event that at each call to
\callSSP{}, the conclusions of Lemma~\ref{lem:SSP} hold, a union bound and 
our choice  
$\delta'=\delta/(2k^2)$ at outer iteration $k$  guarantee that 
\begin{equation*}
\P(\mc{E}) \ge 1 - \sum_{k=1}^\infty\frac{\delta}{2k^2}\ge1-\delta.
\end{equation*}
 We perform our subsequent 
analysis deterministically conditional on the event $\mc{E}$.

Let $f_k$ be the cubic-regularized quadratic model at iteration $k$. We call
the iteration \emph{successful} (Line~\ref{line:progress}) whenever 
\callSSP{} finds
a point $\Delta_k$ such that
\begin{equation*}
  f(\Delta_k)\le -(1-\eps')\rho\slb^3/6 = 
  -\half \left(\frac{\epsilon}{9\rho}\right)^{3/2}\frac{\rho}{6}
  = -\Kprog \epsilon^{3/2}\rho^{-1/2}.
\end{equation*}
The bound~\eqref{eq:majorization} shows that
$g(y_{k-1}+\Delta_k) \le g(y_{k-1}) -\Kprog \epsilon^{3/2}\rho^{-1/2}$ at each
successful iteration, so the last iteration of Algorithm \ref{alg:SP} is
the only unsuccessful one.

Let $\kf$ be the index of the final iteration with model $f_\kf$,
$\siter_\kf = \argmin f_\kf$, $A_\kf = 
\hess g(y_{\kf-1})$, and let $b_\kf = \grad g(y_{\kf-1})$. Since the final 
iteration is unsuccessful, Lemma~\ref{lem:SSP} implies $\|\siter_\kf\| 
\le \sqrt{\epsilon/(9\rho)}$.
Let $\Delta_\kf$ be the point produced by the call to \callSFSP{}, and let $\yout 
= y_\kf + \Delta_\kf$  denote the output of Algorithm~\ref{alg:SP}. Note that 
\callSFSP{} guarantees that $\|\grad f_\kf(\Delta_\kf)\| \le
\epsilon/2$ (we show in the end of this proof that the while loop in 
line~\ref{line:sfsp-while} terminates after a finite number of iterations). 
Moreover, by the same 
argument we use in the proof of Lemma~\ref{lem:SSP}, 
$\eta$ satisfies Assumption 
\ref{assu:step-size}. Since Assumption~\ref{assu:init} 
is also satisfied, we have by Lemma~\ref{lemma:monotone} 
that 
$\|\Delta_\kf\| \le \|\siter_\kf\|$. Therefore, by Assumption~\ref{assu:g} 
we have 
that
\begin{equation*}
  \hess g(\yout) \succeq A_\kf - 2\rho\|\Delta_\kf\|I \succeq -\sqrt{\rho\epsilon}I,
\end{equation*}
where we used $A_\kf \succeq -\rho \|\siter_\kf\| I$ and
$\rho\|\Delta_\kf\|\le\rho\|\siter_\kf\|\le\sqrt{\rho\epsilon}/3$.  That 
is, the
output $\yout$ satisfies the second condition~\eqref{eq:second-order-crit}.

It remains to show that $\nabla g(\yout)$ is small.  Using
$\grad f_\kf(\Delta_\kf) = b_\kf + A \Delta_\kf + \rho
\|\Delta_\kf\|\Delta_\kf$ we
have
\begin{subequations}
\begin{equation}\label{eq:tr-norm-bound-1}
  \|b_\kf + A \Delta_\kf\| \le \|\grad f_\kf(\Delta_\kf)\| + \rho \|\Delta_\kf\|^2 .
\end{equation}
Recalling that $\hess g$ is $2\rho$-Lipschitz 
continuous (Assumption \ref{assu:g}) we have~\cite[Lemma 
1]{NesterovPo06}
\begin{equation}\label{eq:tr-norm-bound-2}
  \left\| \grad g(\yout) - (b_\kf + A \Delta_\kf)\right\| \le \rho 
  \|\Delta_\kf\|^2 \ .
\end{equation}
\end{subequations}
Combining the norm bounds~\eqref{eq:tr-norm-bound-1} 
and~\eqref{eq:tr-norm-bound-2}, and using $\|\grad f_\kf(\Delta_\kf)\| 
\le
\epsilon/2$ and $\rho \|\Delta_\kf\|^2 \le \rho \|\siter_\kf\|^2 \le 
\epsilon/9$ yields
\begin{equation*}
  \|\grad g(\yout)\| \le \|\grad f_\kf(\Delta_\kf)\| + 2\rho \|\Delta_\kf\|^2 \le \epsilon,
\end{equation*}
which completes the proof of $\epsilon$-second-order
stationarity~\eqref{eq:second-order-crit} of $\yout$. 

We now bound the total number of gradient descent iterations
Algorithm~\ref{alg:SP} uses. Noting that $d/\delta' \le 2\Kmax^2 
d/\delta> 1$ and that 
$1/(\eta\rho\slb) > \beta/\sqrt{\rho\epsilon} \ge 1$, we see that a call
to \callSSP{} performs at most $O(1)\beta\rho^{-1/2}\epsilon^{-1/2}( 
\log(d/\delta)+\log\Kmax+\log(\beta/\sqrt{\rho\epsilon}))$ iterations. 
Substituting the bound~\eqref{eq:Kmax-def} on $\Kmax$, the number of 
iterations has the further upper bound
\begin{equation*}
  O(1) \cdot \frac{\beta}{\sqrt{\rho\epsilon}}
  \log\left(\frac{d}{\delta}\cdot\frac{\beta{g(y_0)-\glb}}{\epsilon^{2}}\right).
\end{equation*}
Multiplying this bound by the upper bound on $\Kmax$ shows that the total 
number of steps in all calls \callSSP{} is bounded 
by~\eqref{eq:trustregion-complexity}.

Finally,
standard analysis~\cite[Ex.~1.2.3]{Nesterov04} of gradient descent
on smooth functions shows that \callSFSP{}, which we call exactly once, 
terminates after at most
$2\frac{f(x_0)-f(\siter_\kf)}{\eta (\epsilon/2)^2} \le 
80\beta(g(y_0)-\glb)\epsilon^{-2}$
iterations, as $f_{\kf}$ is $\beta + 2 \rho R$-smooth and
$\eta \le \frac{1}{\beta + 2 \rho R}$.

\bibliographystyle{abbrvnat}

\begin{thebibliography}{34}
\providecommand{\natexlab}[1]{#1}
\providecommand{\url}[1]{\texttt{#1}}
\expandafter\ifx\csname urlstyle\endcsname\relax
  \providecommand{\doi}[1]{doi: #1}\else
  \providecommand{\doi}{doi: \begingroup \urlstyle{rm}\Url}\fi

\bibitem[Agarwal et~al.(2017)Agarwal, Allen-Zhu, Bullins, Hazan, and
  Ma]{AgarwalAlBuHaMa17}
N.~Agarwal, Z.~Allen-Zhu, B.~Bullins, E.~Hazan, and T.~Ma.
\newblock Finding approximate local minima faster than gradient descent.
\newblock In \emph{Proceedings of the Forty-Ninth Annual ACM Symposium on the
  Theory of Computing}, 2017.

\bibitem[Beck and Vaisbourd(2018)]{BeckVa18}
A.~Beck and Y.~Vaisbourd.
\newblock Globally solving the trust region subproblem using simple first-order
  methods.
\newblock \emph{SIAM Journal on Optimization}, 28\penalty0 (3):\penalty0
  1951--1967, 2018.

\bibitem[Bianconcini et~al.(2015)Bianconcini, Liuzzi, Morini, and
  Sciandrone]{BianconciniLiMoSc15}
T.~Bianconcini, G.~Liuzzi, B.~Morini, and M.~Sciandrone.
\newblock On the use of iterative methods in cubic regularization for
  unconstrained optimization.
\newblock \emph{Computational Optimization and Applications}, 60\penalty0
  (1):\penalty0 35--57, 2015.

\bibitem[Bottou et~al.(2018)Bottou, Curtis, and Nocedal]{BottouCuNo18}
L.~Bottou, F.~Curtis, and J.~Nocedal.
\newblock Optimization methods for large-scale learning.
\newblock \emph{SIAM Review}, 60\penalty0 (2):\penalty0 223--311, 2018.

\bibitem[Carmon et~al.(2018)Carmon, Duchi, Hinder, and Sidford]{CarmonDuHiSi18}
Y.~Carmon, J.~C. Duchi, O.~Hinder, and A.~Sidford.
\newblock Accelerated methods for non-convex optimization.
\newblock \emph{SIAM Journal on Optimization}, 28\penalty0 (2):\penalty0
  1751--1772, 2018.

\bibitem[Cartis et~al.(2011{\natexlab{a}})Cartis, Gould, and
  Toint]{CartisGoTo11b}
C.~Cartis, N.~I. Gould, and P.~L. Toint.
\newblock Adaptive cubic regularisation methods for unconstrained optimization.
  {P}art {II}: worst-case function-and derivative-evaluation complexity.
\newblock \emph{Mathematical Programming, Series A}, 130\penalty0 (2):\penalty0
  295--319, 2011{\natexlab{a}}.

\bibitem[Cartis et~al.(2011{\natexlab{b}})Cartis, Gould, and
  Toint]{CartisGoTo11}
C.~Cartis, N.~I.~M. Gould, and P.~L. Toint.
\newblock Adaptive cubic regularisation methods for unconstrained optimization.
  {P}art {I}: motivation, convergence and numerical results.
\newblock \emph{Mathematical Programming, Series A}, 127:\penalty0 245--295,
  2011{\natexlab{b}}.

\bibitem[Cartis et~al.(2012)Cartis, Gould, and Toint]{CartisGoTo12}
C.~Cartis, N.~I. Gould, and P.~L. Toint.
\newblock Complexity bounds for second-order optimality in unconstrained
  optimization.
\newblock \emph{Journal of Complexity}, 28\penalty0 (1):\penalty0 93--108,
  2012.

\bibitem[Conn et~al.(2000)Conn, Gould, and Toint]{ConnGoTo00}
A.~R. Conn, N.~I.~M. Gould, and P.~L. Toint.
\newblock \emph{Trust Region Methods}.
\newblock MPS-SIAM Series on Optimization. SIAM, 2000.

\bibitem[Duchi(2018)]{Duchi18}
J.~C. Duchi.
\newblock Introductory lectures on stochastic convex optimization.
\newblock In \emph{The Mathematics of Data}, IAS/Park City Mathematics Series.
  American Mathematical Society, 2018.

\bibitem[Erway and Gill(2009)]{ErwayGi09}
J.~B. Erway and P.~E. Gill.
\newblock A subspace minimization method for the trust-region step.
\newblock \emph{SIAM Journal on Optimization}, 20\penalty0 (3):\penalty0
  1439--1461, 2009.

\bibitem[Ge et~al.(2015)Ge, Huang, Jin, and Yuan]{GeHuJiYu15}
R.~Ge, F.~Huang, C.~Jin, and Y.~Yuan.
\newblock Escaping from saddle points---online stochastic gradient for tensor
  decomposition.
\newblock In \emph{Proceedings of the Twenty Eighth Annual Conference on
  Computational Learning Theory}, 2015.

\bibitem[Gould et~al.(1999)Gould, Lucidi, Roma, and Toint]{GouldLuRoTo99}
N.~I.~M. Gould, S.~Lucidi, M.~Roma, and P.~L. Toint.
\newblock Solving the trust-region subproblem using the {L}anczos method.
\newblock \emph{SIAM Journal on Optimization}, 9\penalty0 (2):\penalty0
  504--525, 1999.

\bibitem[Gould et~al.(2010)Gould, Robinson, and Thorne]{GouldRoTh10}
N.~I.~M. Gould, D.~P. Robinson, and H.~S. Thorne.
\newblock On solving trust-region and other regularised subproblems in
  optimization.
\newblock \emph{Mathematical Programming Computation}, 2\penalty0 (1):\penalty0
  21--57, 2010.

\bibitem[Griewank(1981)]{Griewank81}
A.~Griewank.
\newblock The modification of {N}ewton's method for unconstrained optimization
  by bounding cubic terms.
\newblock Technical report, Technical report NA/12, 1981.

\bibitem[Hazan and Koren(2016)]{HazanKo16}
E.~Hazan and T.~Koren.
\newblock A linear-time algorithm for trust region problems.
\newblock \emph{Mathematical Programming, Series A}, 158\penalty0 (1):\penalty0
  363--381, 2016.

\bibitem[Kuczynski and Wozniakowski(1992)]{KuczynskiWo92}
J.~Kuczynski and H.~Wozniakowski.
\newblock Estimating the largest eigenvalue by the power and {L}anczos
  algorithms with a random start.
\newblock \emph{SIAM Journal on Matrix Analysis and Applications}, 13\penalty0
  (4):\penalty0 1094--1122, 1992.

\bibitem[LeCun et~al.(2015)LeCun, Bengio, and Hinton]{LeCunBeHi15}
Y.~LeCun, Y.~Bengio, and G.~Hinton.
\newblock Deep learning.
\newblock \emph{Nature}, 521\penalty0 (7553):\penalty0 436--444, 2015.

\bibitem[Lee et~al.(2016)Lee, Simchowitz, Jordan, and Recht]{LeeSiJoRe16}
J.~D. Lee, M.~Simchowitz, M.~I. Jordan, and B.~Recht.
\newblock Gradient descent only converges to minimizers.
\newblock In \emph{Proceedings of the Twenty Ninth Annual Conference on
  Computational Learning Theory}, 2016.

\bibitem[Levy(2016)]{Levy16}
K.~Y. Levy.
\newblock The power of normalization: Faster evasion of saddle points.
\newblock \emph{arXiv:1611.04831 [cs.LG]}, 2016.

\bibitem[Musco and Musco(2015)]{MuscoMu15}
C.~Musco and C.~Musco.
\newblock Randomized block {K}rylov methods for stronger and faster approximate
  singular value decomposition.
\newblock In \emph{Advances in Neural Information Processing Systems}, 2015.

\bibitem[Nesterov(1983)]{Nesterov83}
Y.~Nesterov.
\newblock A method of solving a convex programming problem with convergence
  rate ${O}(1/k^2)$.
\newblock \emph{Soviet Mathematics Doklady}, 27\penalty0 (2):\penalty0
  372--376, 1983.

\bibitem[Nesterov(2004)]{Nesterov04}
Y.~Nesterov.
\newblock \emph{Introductory Lectures on Convex Optimization}.
\newblock Kluwer Academic Publishers, 2004.

\bibitem[Nesterov(2007)]{Nesterov07}
Y.~Nesterov.
\newblock Gradient methods for minimizing composite objective function.
\newblock Technical Report~76, Center for Operations Research and Econometrics
  (CORE), Catholic University of Louvain (UCL), 2007.

\bibitem[Nesterov and Polyak(2006)]{NesterovPo06}
Y.~Nesterov and B.~Polyak.
\newblock Cubic regularization of {N}ewton method and its global performance.
\newblock \emph{Mathematical Programming, Series A}, 108:\penalty0 177--205,
  2006.

\bibitem[Nocedal and Wright(2006)]{NocedalWr06}
J.~Nocedal and S.~J. Wright.
\newblock \emph{Numerical Optimization}.
\newblock Springer, 2006.

\bibitem[Pearlmutter(1994)]{Pearlmutter94}
B.~A. Pearlmutter.
\newblock Fast exact multiplication by the {H}essian.
\newblock \emph{Neural Computation}, 6\penalty0 (1):\penalty0 147--160, 1994.

\bibitem[Polyak(1964)]{Polyak64}
B.~T. Polyak.
\newblock Some methods of speeding up the convergence of iteration methods.
\newblock \emph{{USSR} Computational Mathematics and Mathematical Physics},
  4\penalty0 (5):\penalty0 1--17, 1964.

\bibitem[Schraudolph(2002)]{Schraudolph02}
N.~N. Schraudolph.
\newblock Fast curvature matrix-vector products for second-order gradient
  descent.
\newblock \emph{Neural Computation}, 14\penalty0 (7):\penalty0 1723--1738,
  2002.

\bibitem[Simchowitz et~al.(2018)Simchowitz, Alaoui, and
  Recht]{SimchowitzAlRe18}
M.~Simchowitz, A.~E. Alaoui, and B.~Recht.
\newblock Tight query complexity lower bounds for {PCA} via finite sample
  deformed {W}igner law.
\newblock In \emph{Proceedings of the Fiftieth Annual ACM Symposium on the
  Theory of Computing}, 2018.

\bibitem[Tao and An(1998)]{TaoAn98}
P.~D. Tao and L.~T.~H. An.
\newblock A {D.C.} optimization algorithm for solving the trust-region
  subproblem.
\newblock \emph{SIAM Journal on Optimization}, 8\penalty0 (2):\penalty0
  476--505, 1998.

\bibitem[Trefethen and Bau~III(1997)]{TrefethenBa97}
L.~N. Trefethen and D.~Bau~III.
\newblock \emph{Numerical Linear Algebra}.
\newblock SIAM, 1997.

\bibitem[Weiser et~al.(2007)Weiser, Deuflhard, and Erdmann]{WeiserDeEr07}
M.~Weiser, P.~Deuflhard, and B.~Erdmann.
\newblock Affine conjugate adaptive {N}ewton methods for nonlinear
  elastomechanics.
\newblock \emph{Optimisation Methods and Software}, 22\penalty0 (3):\penalty0
  413--431, 2007.

\bibitem[Zhang et~al.(2017)Zhang, Shen, and Li]{ZhangShLi17}
L.-H. Zhang, C.~Shen, and R.-C. Li.
\newblock On the generalized {L}anczos trust-region method.
\newblock \emph{SIAM Journal on Optimization}, 27\penalty0 (3):\penalty0
  2110--2142, 2017.

\end{thebibliography}

\end{document}